\documentclass[11pt, leqno]{amsart}
\usepackage{a4wide}
\usepackage[mathscr]{eucal}
\usepackage{amsmath}
\usepackage{amssymb}
\usepackage{amsfonts}
\usepackage{amsthm}
\usepackage{color}
\theoremstyle{plain}
\newtheorem{theorem}{Theorem}[section]
\newtheorem{proposition}{Proposition}[section]
\newtheorem{corollary}{Corollary}[section]
\newtheorem{lemma}{Lemma}[section]
\newtheorem{remark}{Remark}

\newtheorem{example}{Example}[section]
\newtheorem{definition}{Definition}[section]

\numberwithin{equation}{section}

\title[Homogeneous Riemannian structures]{Homogeneous structures of $3$-dimensional Sasakian space forms}

\author[J.~Inoguchi]{Jun-ichi Inoguchi}
\address{Department of Mathematics, 
Hokkaido University, 
Sapporo 
060-0810, Japan}
\email{inoguchi@math.sci.hokudai.ac.jp}
\thanks{The first named author is partially supported by JSPS KAKENHI JP19K03461, JP23K03081.}

\author[Y.~Ohno]{Yu Ohno}
\address{Department of Mathematics, Hokkaido University, 
Sapporo, 060-0810, Japan}
\email{ono.yu.h4@elms.hokudai.ac.jp}
\thanks{The second named author is supported by JST SPRING, Grant Number JPMJSP2119.}
\thanks{Tsukuba Journal of Mathematics \textbf{48} (2024), no.~2, 205--247.}

\date{2024.06.20}

\begin{document}

\begin{abstract}
We give explicit parametrizations for all  
homogeneous contact Riemannian structures 
on $3$-dimensional Sasakian space forms.
\end{abstract}


\maketitle

\section*{Introduction}

According to Thurston's classification 
\cite{Thurston} of $3$-dimensional geometries,
there exist eight simply connected model spaces. 
The model spaces are the following homogeneous Riemannian spaces: 
\begin{itemize}
\item space forms: 
Euclidean $3$-space $\mathbb{E}^3$, $3$-sphere 
$\mathbb{S}^3$, hyperbolic $3$-space $\mathbb{H}^3$,
\item reducible Riemannian symmetric spaces: 
$\mathbb{S}^2\times \mathbb{R}$, 
$\mathbb{H}^2\times \mathbb{R}$, 
\item the Heisenberg group $\mathrm{Nil}_3$,
the universal covering group 
$\widetilde{\mathrm{SL}}_2\mathbb{R}\cong \widetilde{\mathrm{SU}}(1,1)$ of 
the special linear group $\mathrm{SL}_2\mathbb{R}\cong\mathrm{SU}(1,1)$.
\item the Minkowski motion group $\mathrm{Sol}_3$.
\end{itemize}  
Other than space forms and product spaces, 
model space are \emph{not} Riemannian symmetric spaces.
As is well known, local symmetry of Riemannian manifolds is 
characterized by the parallelism of the Riemannian 
curvature due to Cartan. 
As a generalization of local symmetry, 
Ambrose and Singer \cite{AS} obtained an infinitesimal
characterization of Riemannian homogeneity of Riemannian manifolds.
They showed that local Riemannian homogeneity is equivalent to the existence of certain tensor field.
Such a tensor field is referred as to a \emph{homogeneous Riemannian structure}.
The moduli space of homogeneous Riemannian structures on a homogeneous Riemannian space 
$(M,g)$ represents roughly the all possible coset space representations of $(M,g)$  
up to isomorphisms. In other words, the moduli space may be regarded as 
the space of \emph{canonical connections} (also called the \emph{Ambrose-Singer connections}). 
Note that Ambrose-Singer connections are 
canonical connections in the sense of Olmos and Sanchez \cite{OS}. 
Katsuda \cite{Katsuda} obtained a pinching theorem for locally 
homogenous Riemannian spaces by using homogeneous Riemannian structures.

\medskip

Among the eight model spaces,
$\mathbb{S}^3$, 
$\mathrm{Nil}_3$, $\widetilde{\mathrm{SL}}_2\mathbb{R}$ 
and $\mathrm{Sol}_3$ admit homogeneous contact structure compatible to the 
metric. In particular, $\mathbb{S}^3$, 
$\mathrm{Nil}_3$, $\widetilde{\mathrm{SL}}_2\mathbb{R}$ are Sasakian space forms.
Note that 3-dimensional simply connected Sasakian space forms are exhausted by
$\mathbb{S}^3$, 
$\mathrm{Nil}_3$, $\widetilde{\mathrm{SL}}_2\mathbb{R}$ and the Berger $3$-spheres.
Moreover Sasakian space forms are \emph{naturally reductive} homogeneous Riemannian spaces 
\cite{BV2}. It should be remarked that all simply connected naturally reductive homogeneous 
$3$-spaces are exhausted by Riemannian symmetric spaces and Sasakian space forms (see \cite{TV}).
Naturally reductive homogeneous spaces have some particular properties. 
For instance, every geodesic symmetry is volume preserving up to sign.

This article has two purposes. 
The first purpose of the present article is to determine all homogeneous Riemannian structures 
as well as homogeneous contact Riemannian structures on $3$-dimensional Sasakian space forms.
We prove that all homogeneous Riemannian structures on a $3$-dimensional Sasakian space 
form constitute a one-parameter family of Ambrose-Singer connections.

\smallskip

In contact Riemannian geometry or CR-geometry, 
certain kinds of linear connections with non-vanishing torsion have been used. 
Tanaka \cite{TanakaBook} 
and Webster \cite{Web} introduced a 
linear connection on contact strongly pseudo-convex CR-manifolds. 
This connection is referred as to the \emph{Tanaka-Webster connection}. 
Note that Sasakian manifolds (of arbitrary odd-dimension) are strongly pseudo-convex CR-manifolds. 
The Tanaka-Webster connection is a metrical connection 
with non-vanishing torsion.

On the other hand, Okumura \cite{Ok} introduced a 
one-parameter family of metrical connections with non-vanishing torsion on 
Sasakian manifolds of arbitrary odd-dimension. One can see that 
Okumura's family includes Tanaka-Webster connection.  

The second purpose of the present article is to study 
relations between these two kinds of connections-- Ambrose-Singer connections and 
Okumura's family of connections (including Tanaka-Webster connection)-- 
derived from different geometric backgrounds. 
We prove that the set of all Ambrose-Singer connections 
on each $3$-dimensional Sasakian space form 
coincides with the Okumura's one-parameter family if 
the holomorphic sectional curvature is equal or greater than $-3$ and not equal to $1$. In the case holomorphic sectional curvature is less than $-3$, there 
the set of all Ambrose-Singer connections consists 
of Okumura's one-parameter family and the 
canonical connection of a certain $3$-dimensional solvable Lie group.

Throughout this paper all manifolds are assumed to be connected.

The authors would like to express their sincere thanks to the referee for her/his 
careful reading of the manuscript and suggestions for improvements 
of the original manuscript. In particular, Lemma \ref{lem:sigma} is suggested by the referee. 


\section{Homogeneous Riemannian spaces}\label{sec:2}
\subsection{}
Let $(M,g)$ be a Riemannian manifold with Levi-Civita connection
$\nabla$. The \emph{Riemannian curvature} $R$ is defined by
\[
R(X,Y)=[\nabla_{X},\nabla_{Y}]-\nabla_{[X,Y]}.
\]
The Ricci tensor field $\mathrm{Ric}$ is defined by
\[
\mathrm{Ric}(X,Y)=\mathrm{tr}_{g}(Z\longmapsto R(Z,Y)X).
\]
\begin{definition}
{\rm
A Riemannian manifold $(M,g)$ is said to be a 
\emph{homogeneous Riemannian space} if there exists a Lie group $G$ of 
isometries which acts transitively on $M$.

More generally, $M$ is said to be \emph{locally homogeneous Riemannian 
space} if for each $p$, $q\in M$, 
there exists a local isometry which sends $p$ to 
$q$.
}
\end{definition}

Without loss of generality, we may assume that 
a homogeneous Riemannian space $(M,g)$ is reductive. Indeed, the following result is 
known (see \textit{e.g.}, \cite{KoSz}):
\begin{proposition}
Any homogenous Riemannian space $(M,g)$ is a reductive homogenous space.
\end{proposition}

\subsection{}
Ambrose and Singer \cite{AS} 
gave an \emph{infinitesimal characterization} of 
local homogeneity of Riemannian manifolds. 
To explain their characterization we recall the following notion:

\begin{definition}{\rm
A \emph{homogeneous Riemannian structure} $S$ on $(M,g)$ is
a tensor field of type $(1,2)$ which satisfies
\begin{equation}
\tilde{\nabla}{g}=0,
\quad 
\tilde{\nabla}{R}=0,
\quad 
\tilde{\nabla}{S}=0.
\end{equation}
Here $\tilde{\nabla}$ is a linear connection on $M$ defined
by $\tilde{\nabla}=\nabla+S$. The linear connection $\tilde{\nabla}$ is called 
the \emph{Ambrose-Singer connection}.
}
\end{definition}
\begin{remark}{\rm
The sign conventions of $R$ and $S$ are opposite 
to the ones used in \cite{TV,CalLo}. 
}
\end{remark}

Let $(M,g)=G/H$ be a homogeneous Riemannian space.
Here $G$ is a connected Lie group acting transitively
on $M$ as a group of isometries. 
Without loss of generality we can assume that $G$ acts 
\emph{effectively} on $M$. 

The subgroup $H$ is the isotropy subgroup
of $G$ at a point $o\in M$ which will be called the \emph{origin} of $M$. 
For any $a\in G$, 
the translation $\tau_{a}$ by $a$ is 
a diffeomorphism on $M$ defined by $\tau_{a}(bH)=(ab)H$. 
Denote by $\mathfrak{g}$ and $\mathfrak{h}$ the Lie algebras of $G$ and
$H$, respectively. Then there exists a linear subspace 
$\mathfrak{m}$ of $\mathfrak{g}$ which is $\mathrm{Ad}(H)$-invariant.
If $H$ is connected, then $\mathrm{Ad}(H)$-invariant property 
of $\mathfrak{m}$ is equivalent to
the condition $[\mathfrak{h},\mathfrak{m}]\subset
\mathfrak{m}$, \textit{i.e.}, $G/H$ is reductive.
The tangent space 
$\mathrm{T}_{p}M$ of $M$ at a point 
$p=\tau_a(o)$ is identified with $\mathfrak{m}$ via the isomorphism
\[
\mathfrak{m}\ni
X
\longleftrightarrow 
X^{*}_{p}=\frac{\mathrm{d}}{\mathrm{d}t}\biggr \vert_{t=0}
\tau_{\exp(tX)}(p)
\in\mathrm{T}_{p}M.
\]
Then the canonical connection 
$\tilde{\nabla}=\nabla^{\mathrm c}$ is given by
\[
(\tilde{\nabla}_{X^*}Y^{*})_o
=-([X,Y]_{\mathfrak m})^{*}_o,
\quad 
X,Y 
\in \mathfrak{m}.
\]
For any vector $X\in\mathfrak{g}$, 
the $\mathfrak{m}$-component of $X$ is 
denoted by $X_{\mathfrak{m}}$.  
One can see the difference tensor field $S=\tilde{\nabla}-\nabla$
is a homogeneous Riemannian structure.
Thus every homogeneous Riemannian space admits homogeneous Riemannian structures.

Conversely, let $(M,S)$ be a simply connected Riemannian manifold with a
homogeneous Riemannian structure. Fix a point $o\in M$ and put 
$\mathfrak{m}=\mathrm{T}_{o}M$. Denote by $\tilde{R}$ the curvature of the 
Ambrose-Singer connection $\tilde{\nabla}$. Then the holonomy algebra
$\mathfrak{h}$ of $\tilde{\nabla}$ 
is the Lie subalgebra of the Lie algebra 
$\mathfrak{so}(\mathfrak{m},g_o)$ 
generated by the curvature 
operators $\tilde{R}(X,Y)$ with 
$X$, $Y\in\mathfrak{m}$.

Now we define a Lie algebra structure on the direct sum 
$\mathfrak{g}=\mathfrak{h}\oplus \mathfrak{m}$
by (see \cite{AS,TV})
\begin{align*}
[U,V]=& UV-VU,
\notag \\
[U,X]=& U(X),
\notag 
\\
[X,Y]=& -\tilde{R}(X,Y)-S(X)Y+S(Y)X
\end{align*}
for all $X$, $Y\in \mathfrak{m}$ and $U$, $V\in \mathfrak{h}$.

Now let $\tilde{G}$ be the simply connected Lie group with
Lie algebra $\mathfrak{g}$. Then $M$ is a coset manifold $\tilde{G}/\tilde{H}$, where
$\tilde{H}$ is a Lie subgroup of $\tilde{G}$ with Lie algebra $\mathfrak{h}$.
Let $\Gamma$ be the set of all elements in $G$ which act trivially on $M$. Then
$\Gamma$ is a discrete normal subgroup of $\tilde{G}$ and $G=\tilde{G}/\Gamma$
acts transitively and effectively on $M$ as an isometry group.
The isotropy subgroup $H$ of $G$ at $o$ is $H=\tilde{H}/\Gamma$.  
Hence $(M,g)$ is a homogeneous Riemannian space with coset
space representation $M=G/H$.

\begin{theorem}[\cite{AS}]
A Riemannian manifold $(M,g)$ with a homogeneous Riemannian structure $S$
is locally homogeneous. 
\end{theorem}

\subsection{}
For a homogeneous Riemannian structure $S$ on a Riemannian manifold 
$(M,g)$, we denote by $S_{\flat}$ the covariant tensor field metrically 
equivalent to $S$, that is,
\[
S_{\flat}(X,Y,Z)=g(S(X)Y,Z)
\]
for all $X$, $Y$, $Z\in \varGamma(\mathrm{T}M)$.
The metrical condition 
$\tilde{\nabla}g=0$ is rewritten as
\begin{equation}\label{eq:AS-met}
S_{\flat}(X,Y,Z)+S_{\flat}(X,Z,Y)=0
\end{equation}
for all vector fields $X$, $Y$ and $Z$.

Tricerri and Vanhecke \cite{TV} obtained the following decompositions 
of all possible types of homogeneous Riemannian structures into eight classes:
\begin{center}
\begin{tabular}{|l|l|}
\hline
Classes & Defining conditions  \\
\hline
Symmetric   & $S=0$    \\
\hline
$\mathcal{T}_1$   & $S_{\flat}(X,Y,Z)=g(X,Y)\omega(Z)-g(Z,X)\omega(Y)$  for some $1$-form $\omega$\\
\hline
$\mathcal{T}_2$ & $\underset{X,Y,Z}{\mathfrak{S}}\, S_{\flat}(X,Y,Z)=0$ and $c_{12}(S_\flat)=0$ \\
\hline
$\mathcal{T}_3$ & $S_{\flat}(X,Y,Z)+S_{\flat}(Y,X,Z)=0$  \\
\hline
$\mathcal{T}_{1}\oplus \mathcal{T}_2$ & $\underset{X,Y,Z}{\mathfrak{S}}\, S_{\flat}(X,Y,Z)=0$ \\
\hline
$\mathcal{T}_{1}\oplus \mathcal{T}_3$
&
$
S_{\flat}(X,Y,Z)+S_{\flat}(Y,X,Z)=2g(X,Y)\omega(Z)-g(Z,X)\omega(Y)-g(Y,Z)\omega(X)
$
\\
{} &
for some $1$-form $\omega$\\
\hline  
$\mathcal{T}_2\oplus \mathcal{T}_3$ & $c_{12}(S_\flat)=0$ \\
\hline
$\mathcal{T}_{1}\oplus\mathcal{T}_2\oplus \mathcal{T}_3$
& no conditions
\\
\hline
\end{tabular}
\end{center} 

\medskip

Here $\underset{X,Y,Z}{\mathfrak{S}}\,S_{\flat}$ denotes the cyclic sum of $S_\flat$, \textit{i.e.},
\[
\underset{X,Y,Z}{\mathfrak{S}}\, S_{\flat}(X,Y,Z)=
S_{\flat}(X,Y,Z)+S_{\flat}(Y,Z,X)+S_{\flat}(Z,X,Y).
\]
Next $c_{12}$ denotes the contraction operator in $(1,2)$-entries;
\[
c_{12}(S_{\flat})(Z)=\sum_{i=1}^{n}S_{\flat}(e_i,e_i,Z),
\] 
where $\{e_1, e_2, \dots ,e_n\}$ is an arbitrary local orthonormal frame field.

\subsection{}

A reductive homogeneous Riemannian space $M=G/H$ with 
Lie subspace $\mathfrak{m}$ is said to be 
\emph{naturally reductive} if 
every geodesic through $o\in M$ and tangent to $X\in \mathfrak{m}=\mathrm{T}_{o}M$
is the {orbit of the one-parameter subgroup $\{\exp(tX)\}$.

The following infinitesimal reformulation of natural reducibility is 
useful.

\begin{proposition}
A reductive homogeneous Riemannian space $M=G/H$ with 
Lie subspace $\mathfrak{m}$ is naturally reductive if and only if
\[
\langle [X,Y]_{\mathfrak m},Z
\rangle+
\langle Y,[X,Z]_{\mathfrak m}
\rangle
=0
\]
for all $X$, $Y$, $Z\in \mathfrak{m}$. 
Here $\langle\cdot,\cdot\rangle$ is the inner product 
of $\mathfrak{m}$ induced from the Riemannian metric $g$ of $M$.
\end{proposition}

It should be remarked that the notion of 
``naturally reductive" depends on the choice of the Lie subspace $\mathfrak{m}$. 
To avoid the ambiguity, the notion of naturally reductive space is introduced as
follows:

\begin{definition}{\rm
A reductive homogeneous Riemannian space $M=G/H$ is said to be 
\emph{naturally reductive} if it admits a Lie subspace $\mathfrak{m}$ 
such that $G/H$ is naturally reductive with respect to $\mathfrak{m}$.
}
\end{definition}

Tricerri and Vanhecke obtained the following 
characterizations.

\begin{theorem}[\cite{TV}, \cite{TV2}]
Let $M=G/H$ be a 
naturally reductive homogeneous space with canonical 
connection $\tilde{\nabla}$. 
Then $S=\tilde{\nabla}-\nabla$ is a homogeneous Riemannian 
structure which satisfies 
\begin{equation}\label{AST3}
S(X)X=0
\end{equation}
for all $X\in \varGamma(\mathrm{T}M)$.
Namely $S$ is of type $\mathcal{T}_3$.
Conversely, a simply connected and complete Riemannian manifold
$(M,g,S)$ together with a homogeneous structure
$S$ is naturally reductive if and
only if $S$ satisfies \eqref{AST3}.
\end{theorem}

\begin{theorem}[\cite{TV}]
Let $(M,g)$ be a complete and 
simply connected Riemannian manifold.  
Then $M$ is a naturally reductive homogeneous space 
if and only if there exists a homogeneous Riemannian structure 
$S$ so that $\tilde{\nabla}=\nabla+S$ is projectively equivalent to 
the Levi-Civita connection $\nabla$.
\end{theorem}
Tricceri and Vanhecke \cite{TV} classified 
homogeneous Riemannian $3$-manifolds 
admitting homogeneous Riemannian structure of type $\mathcal{T}_3$. 
On the other hand, Kowalski and Tricceri \cite{KoTr} 
classified 
homogeneous Riemannian $3$-manifolds 
admitting homogeneous Riemannian structure of type $\mathcal{T}_2$.

\subsection{}
Olmos and S{\'a}mchez \cite{OS} introduced the 
notion of ``canonical connection" in 
the following manner:
\begin{definition}{\rm 
Let $(M,g)$ be a Riemannian manifold with 
Levi-Civita connection. A metric linear connection 
$\nabla^{\mathsf c}$ is said to be a 
\emph{canonical connection} in the sense of Olmos-S{\'a}nchez 
if it satisfies $\nabla^{\mathsf c}S=0$, where 
$S=\nabla^{\mathsf c}-\nabla$.
}
\end{definition}
Obviously, any Ambrose-Singer connections are 
canonical connections in the sense of Olmos-S{\'a}mchez. 
Under this definition, they proved the following result.
\begin{theorem}[\cite{OS}]
Let $M$ be linearly full 
immersed submanifold of Euclidean space 
$\mathbb{E}^n$ with vector valued second fundamental form 
$\mathrm{I\!I}$. Then the following properties are 
mutually equivalent{\rm:}
\begin{enumerate}
\item $M$ admits a canonical connection 
in the sense of Olmos-S{\'a}nchez satisfying 
$\nabla^{\mathsf c}\mathrm{I\!I}=0$.
\item $M$ is an (extrinsically) homogeneous submanifold 
with constant principal curvatures.
\item $M$ is an orbit of an $s$-representation. Namely $M$ is a 
standardly embedded $R$-spaces.
\end{enumerate} 
\end{theorem}

\subsection{Canonical connections on Lie groups}\label{sec:1.6new}
We can regard a Lie group $G$ as a homogeneous space in two ways: 
$G=G/\{\mathsf{e}\}$ and $G=(G\times G)/\Delta G$. 
Here $\mathsf{e}$ is the unit element of $G$. 
In the first representation, the isotropy algebra is 
$\{0\}$ and the tangent space $\mathfrak{m}$ at $\mathsf{e}$ is identified with
$\mathfrak{g}$. 
Obviously the splitting $\mathfrak{g}=\{0\}+
\mathfrak{g}$ is reductive. The natural projection 
$\pi:G\to G/\{\mathsf e\}$ is the identity map.
The canonical connection of $G/\{\mathsf{e}\}$ is 
denoted by $\nabla^{(-)}$ and given by 
\[
\nabla^{(-)}_{X}Y=0,\quad X,Y\in\mathfrak{g}.
\]
The torsion $T^{(-)}$ of $\nabla^{(-)}$ is given by 
$T^{(-)}(X,Y)=-[X,Y]$. The canonical connection 
$\nabla^{(-)}$ is also called the \emph{Cartan-Schouten's $(-)$-connection}.

Next, let us take the product Lie group $G\times G$. 
The Lie algebra of $G\times G$ is 
\[
\mathfrak{g}\oplus\mathfrak{g}=\{(X,Y)\>|\>X,Y\in\mathfrak{g}\}
\]
with Lie bracket
\[
[(X_1,Y_1),(X_2,Y_2)]=([X_1,Y_1],[X_2,Y_2]).
\]
The product Lie group $G\times G$ 
acts on $G$ by the action:
\begin{equation}\label{eq:biinvariantaction}
(G\times G)\times G\to G;\quad (a,b)x=axb^{-1}.
\end{equation}
The isotropy subgroup at the identity $\mathsf{e}$ is the 
\emph{diagonal subgroup} 
\[
\Delta G=\{(a,a)\>|\>a\in G\} 
\]
with Lie algebra $\Delta\mathfrak{g}=\{(X,X)\>|\>X\in\mathfrak{g}\}$. 
We can consider the following three 
Lie subspaces;
\[
\mathfrak{m}^{+}=\{(0,X)\>|\>X\in\mathfrak{g}\},
\quad 
\mathfrak{m}^{-}=\{(X,0)\>|\>X\in\mathfrak{g}\},
\quad 
\mathfrak{m}^{0}=\{(X,-X)\>|\>X\in\mathfrak{g}\}.
\]
Then $\mathfrak{g}\oplus\mathfrak{g}=\Delta\mathfrak{g}\oplus \mathfrak{m}^{+}$,
$\mathfrak{g}\oplus\mathfrak{g}=\Delta\mathfrak{g}\oplus \mathfrak{m}^{-}$ 
and 
$\mathfrak{g}\oplus\mathfrak{g}=\Delta\mathfrak{g}\oplus \mathfrak{m}^{0}$ are reductive.

The corresponding splittings are given 
explicitly by
\begin{align*}
&(X,Y)=(X,X)+(0,-X+Y)\in \Delta\mathfrak{g}\oplus \mathfrak{m}^{+}
\\
&(X,Y)=(Y,Y)+(X-Y,0)\in \Delta\mathfrak{g}\oplus \mathfrak{m}^{-}
\\
&(X,Y)=\left(\frac{X+Y}{2},\frac{X+Y}{2}\right)+
\left(\frac{X-Y}{2},-\frac{X-Y}{2}\right)\in \Delta\mathfrak{g}\oplus \mathfrak{m}^{0}.
\end{align*}
Let us identify the tangent space 
$\mathrm{T}_{\mathsf{e}}G$ of $G$ at $\mathsf{e}$ with these 
Lie subspaces. Then the canonical connection with respect to the 
reductive decomposition 
$\mathfrak{g}\oplus\mathfrak{g}=\Delta\mathfrak{g}\oplus \mathfrak{m}^{+}$
is denoted by $\nabla^{(+)}$ and given by
\[
\nabla^{(+)}_{X}Y=[X,Y],\quad X,Y\in\mathfrak{g}.
\]
The torsion $T^{(+)}$ of $\nabla^{(+)}$ is given by 
$T^{(+)}(X,Y)=[X,Y]$. 
The canonical connection 
$\nabla^{(+)}$ is called the \emph{Cartan-Schouten's $(+)$-connection} 
or \emph{anti canonical connection} \cite{DITAM}.

Next, the 
canonical connection with respect to the 
reductive decomposition 
$\mathfrak{g}\oplus\mathfrak{g}=\Delta\mathfrak{g}\oplus \mathfrak{m}^{-}$ is 
$\nabla^{(-)}$. Finally  
the canonical connection with respect to the 
reductive decomposition 
$\mathfrak{g}\oplus\mathfrak{g}=\Delta\mathfrak{g}\oplus \mathfrak{m}^{0}$
is denoted by $\nabla^{(0)}$ and given by
\[
\nabla^{(0)}_{X}Y=\frac{1}{2}[X,Y],\quad X,Y\in\mathfrak{g}.
\]
The connection $\nabla^{(0)}$ is torsion free and called 
the \emph{Cartan-Schouten's $(0)$-connection}, 
the \emph{natural torsion free connection} \cite{Nomizu} or  
\emph{neutral connection} \cite{DITAM}.

\subsection{The homogeneous Riemannian structures on the 
hyperbolic plane}
Let us realize the hyperbolic plane 
$\mathbb{H}^2(-\alpha^2)$ of constant curvature $-\alpha^2<0$~($\alpha>0$)~
as the upper half plane 
\[
\{z=x+yi\in\mathbb{C}\>|\>y>0\}
\]
equipped with the Poincar{\'e} metric 
\[
\frac{\mathrm{d}x^2+\mathrm{d}y^2}{\alpha^2y^2}.
\]
The special linear group $\mathrm{SL}_2\mathbb{R}$ acts 
isometrically and transitively on $\mathbb{H}^2(-\alpha^2)$ via the 
linear fractional action:
\[
\left(
\begin{array}{cc}
a & b\\
c & d
\end{array}
\right)\cdot z=\frac{az+b}{cz+d}.
\]
The isotropy subgroup at $i$ is the rotation group 
$\mathrm{SO}(2)$. The natural projection 
$\pi:\mathrm{SL}_2\mathbb{R}\to\mathbb{H}^2(-\alpha^2)$ is 
given explicitly by
\[
\pi\left(\>\left(
\begin{array}{cc}
a & b\\
c & d
\end{array}
\right)\>\right)=\frac{(ac+bd)+i}{c^2+d^2}.
\]
The Iwasawa decomposition 
of $\mathrm{SL}_2\mathbb{R}$ is 
carried out as 
$\mathrm{SL}_2\mathbb{R}=NAK$, where

\begin{align*}
N=&
\left\{
\left.
\left(
\begin{array}{cc}
1 & x\\
0 & 1
\end{array}
\right)
\>\right|\>
x\in\mathbb{R}
\right\}
\cong (\mathbb{R},+),
\\
A=&
\left\{
\left.
\left(
\begin{array}{cc}
\sqrt{y} & 0\\
0 & 1/\sqrt{y}
\end{array}
\right)
\>\right|\>
y>0
\right\}
\cong \mathrm{SO}^{+}(1,1),
\\
K=&
\left\{
\left.
\left(
\begin{array}{cc}
\cos\phi & \sin\phi\\
-\sin\phi & \cos\phi
\end{array}
\right)
\>\right|\>
\phi\in\mathbb{R}/2\pi\mathbb{Z}
\right\}
= \mathrm{SO}(2),
\end{align*}
We refer $(x,y,\phi)$ as a coordinate system 
of $\mathrm{SL}_2\mathbb{R}$. 
We can take an orthonormal basis 
\[
\bar{e}_0=
\frac{1}{\alpha}
\left(
\begin{array}{rr}
0 & 1\\
-1 & 0
\end{array}
\right),
\quad 
\bar{e}_1=
\frac{1}{\alpha}
\left(
\begin{array}{rr}
0 & 1\\
1 & 0
\end{array}
\right),
\quad 
\bar{e}_2=
\frac{1}{\alpha}
\left(
\begin{array}{rr}
1 & 0\\
0 & -1
\end{array}
\right)
\]
of the Lie algebra $\mathfrak{sl}_2\mathbb{R}$.
The tangent space $\mathrm{T}_{i}\mathbb{H}^2(-\alpha^2)$ 
is identified with the linear subspace $\mathfrak{p}$ spanned by 
$\bar{e}_1$ and $\bar{e}_2$. On the other hand, 
the Lie algebra $\mathfrak{k}$ of the isotropy subgroup $\mathrm{SO}(2)$ is 
spanned by $\bar{e}_0$. 
One can see that $[\mathfrak{k},\mathfrak{p}]
\subset \mathfrak{p}$ and 
$[\mathfrak{p},\mathfrak{p}]
\subset \mathfrak{k}$. Hence 
$\mathbb{H}^2(-\alpha^2)=\mathrm{SL}_2\mathbb{R}
/\mathrm{SO}(2)$ is a Riemannian symmetric space. 

The upper half plane model 
$\mathbb{H}^2(-\alpha^2)$ is identified with 
the solvable Lie group 
\[
\mathcal{S}=NA=
\left\{
\left.
\left(
\begin{array}{cc}
\sqrt{y} & x/\sqrt{y}\\
0 & 1/\sqrt{y}
\end{array}
\right)
\>\right|\>
(x,y)\in\mathbb{H}^2(-\alpha^2)
\right\}.
\]
The Lie algebra $\mathfrak{s}$ of $S$ is 
spanned by $\bar{e}_0+\bar{e}_1$ and $\bar{e}_2$. 
The tangent space $\mathrm{T}_{i}\mathbb{H}^2(-\alpha^2)$ is 
identified with 
the Lie algebra $\mathfrak{s}$ under the isomorphism:
\[
y\frac{\partial}{\partial x}\biggr\vert_{i}\longmapsto
\bar{e}_0+\bar{e}_1,
\quad
y\frac{\partial}{\partial y}\biggr\vert_{i}\longmapsto
\bar{e}_2.
\]
Hence $\mathbb{H}^2(-c^2)$ is regarded as 
a solvable Lie group 
with respect to the multiplication
\[
(x_1,y_1)(x_2,y_2)=
(x_1+y_1x_2,y_1y_2).
\]
It should be remarked that 
$\mathfrak{sl}_2\mathbb{R}$ is decomposed as 
$\mathfrak{sl}_2\mathbb{R}=\mathfrak{k}+\mathfrak{s}$.
However this decomposition is \emph{not} reductive.

\begin{theorem}[\cite{TV}]\label{thm:1.5}
The homogeneous Riemannian structures 
of $\mathbb{H}^2(-\alpha^2)$ are classified as 
follows{\rm:}
\begin{itemize}
\item $S=0:$ The corresponding homogeneous space representation is 
$\mathrm{SL}_2\mathbb{R}/\mathrm{SO}(2)$.
\item $S(X)Y=-g(X,Y)\bar{e}_2+g(\bar{e}_2,Y)X$ for 
$X$, $Y\in\varGamma(\mathrm{T}\mathbb{H}^2(-\alpha^2)):$ 
This homogeneous Riemannian structure is of 
type $\mathcal{T}_1$. 
The corresponding homogeneous space representation is 
$\mathcal{S}/\{\mathsf{e}\}$. The connection 
$\nabla+S$ is the Cartan-Schouten's $(-)$-connection 
of $\mathcal{S}$.
\end{itemize}
\end{theorem}
The solvable Lie group 
$\mathcal{S}$ is isomorphic to the 
identity component 
$\mathrm{GA}^{+}(1)=\mathrm{GL}^{+}_{1}
\mathbb{R}\ltimes\mathbb{R}$ of the 
affine transformation group 
$\mathrm{GA}(1)=\mathrm{GL}_{1}
\mathbb{R}\ltimes\mathbb{R}$ of the real line. 
Indeed, 
\[
\left(
\begin{array}{cc}
\sqrt{y} & x/\sqrt{y}\\
0 & 1/\sqrt{y}
\end{array}
\right)\longmapsto 
\left(
\begin{array}{cc}
y & x\\
0 &1
\end{array}
\right)
\]
gives the Lie group isomorphism 
from $\mathcal{S}$ onto 
$\mathrm{GA}^{+}(1)$. 
The orthonormal basis 
$\{\bar{e}_1,\bar{e}_2\}$
of $\mathfrak{s}$ corresponds 
to the orthonormal basis
\[
\bar{e}_1=\left(
\begin{array}{cc}
0 & \alpha\\
0 & 0
\end{array}
\right),
\quad 
\bar{e}_2=\left(
\begin{array}{cc}
\alpha & 0\\
0 & 0
\end{array}
\right)
\]
of the Lie algebra $\mathfrak{ga}(1)$. 
\begin{remark}\label{rem:H2}
{\rm The only homogeneous Riemannian structure 
of the $2$-sphere $\mathbb{S}^2(\alpha^2)$ is 
the trivial one ($S=0$). The hyperbolic plane 
$\mathbb{H}^2(-\alpha^2)$ has non trivial homogeneous Riemannian 
structure of type $\mathcal{T}_1$. This difference 
of $\mathbb{S}^2(\alpha^2)$ and $\mathbb{H}^2(-\alpha^2)$ reflects 
on the diference of homogeneous structures of 
$\mathbb{S}^2(\alpha^2)\times\mathbb{R}$ and 
$\mathbb{H}^2(-\alpha^2)\times\mathbb{R}$ (see \cite{Ohno}). 
Moreover the homogeneous structure on $\mathbb{H}^2(-\alpha^2)$ 
of type $\mathcal{T}_1$ induces a homogeneous Riemannian structure 
on the Sasakian space form $\mathscr{M}^{3}(-3-\alpha^2)\times\mathbb{R}
=\mathbb{H}^2(-\alpha^2)\times\mathbb{R}$ (see Example \ref{eg:hyperbolicSasakian}). 
}
\end{remark}
\section{Moving frames}
For later use, here we collect fundamental equations 
of moving frames on Riemannian $3$-manifolds.
\subsection{Connection forms}
Let $(M,g)$ be a Riemannian $3$-manifold. Take a local orthonormal frame field 
$\mathcal{E}=\{e_1,e_2,e_3\}$. Denote by
$\Theta=(\theta^1,\theta^2,\theta^3)$ 
the orthonormal coframe field 
metrically dual to $\{e_1,e_2,e_3\}$. 
We regard $\Theta$ as a vector valued $1$-form 
\[
\Theta=\left(
\begin{array}{c}
\theta^1\\
\theta^2\\
\theta^3
\end{array}
\right).
\]
Since the 
Levi-Civita connection is torsion free, 
the following 
\emph{first structure equation}
\[
\mathrm{d}\Theta+\omega\wedge \Theta=0.
\]
holds.
The $\mathfrak{so}(3)$-valued $1$-form  
\[
\mathbf{\omega}=
\left(
\begin{array}{ccc}
0 & \omega_{2}^{\>\,1} & \omega_{3}^{\>\,1}
\\
-\omega_{2}^{\>\,1} & 0 & \omega_{3}^{\>\,2}\\
- \omega_{3}^{\>\,1} & -\omega_{3}^{\>\,2} & 0 
\end{array}
\right)
\]
determined by the first
structure equation is called the 
\emph{connection form}.  
A component $\omega_{j}^{\>i}$ of $\omega$ is 
called a connection $1$-form. 
The first structure equation is 
the differential system:
\[
\mathrm{d}\theta^{i}+\sum_{j=1}^{3}\omega_{j}^{\>\,i}\wedge \theta^{j}=0.
\]
The connection coefficients 
$\{\varGamma_{jk}^{\,\>i}\}$ of 
the Levi-Civita connection $\nabla$ is 
relative to $\mathcal{E}$ is defined by
\[
\nabla_{e_i}e_{j}=\sum_{k=1}^{3}\varGamma_{ij}^{\,\,k}e_k.
\]
Then the connection $1$-forms are related to  
connection coefficients by
\[
\omega_{j}^{\>k}=
\sum_{\ell=1}^{3}\varGamma_{\ell j}^{\,\,k}\,\theta^{\ell}.
\]
Hence we obtain
\begin{equation}\label{eq:LC-omega}
g(\nabla_{X}e_i,e_j)=
\omega_{i}^{\>j}(X).
\end{equation}
Thus 
\[
\omega_{i}^{\>\>j}=-\omega_{j}^{\>\>i}.
\]
\begin{remark}
{\rm Tricerri and Vanhecke \cite{TV} used the convention:
\[
g(\nabla_{X}e_i,e_j)=
\omega_{ij}(X).
\]
}
\end{remark}
\subsection{Curvature forms}
Next, 
the $\mathfrak{so}(3)$-valued 
$2$-form $\varOmega=(\varOmega_{j}^{\>\,i})$ defined by
\[
\varOmega=\mathrm{d}\omega+\omega\wedge\omega
\]
is called the 
\emph{curvature form} relative to $\Theta$. 
This formula is called the 
\emph{second structure equation}. 
The components $\varOmega_{j}^{\>\>i}$ are called 
curvature $2$-forms. The second structure equation is the 
differential system:
\[
\varOmega_{j}^{\>\,i}=
\mathrm{d}\omega_{j}^{\>\,i}+\sum_{k=1}^{3}
\omega_{k}^{\>\,i}\wedge
\omega_{j}^{\>\,k}.
\]
One can see 
that 
\[
R(X,Y)e_{i}=2\sum_{i=1}^{3}\varOmega_{i}^{\>\>j}(X,Y)e_{j}.
\]
If we express the Riemannian curvature $R$ as
\[
R(e_k,e_{\ell})e_i=\sum_{j=1}^{3}
R_{ik\ell}^{\>\>j}\,e_{j},
\]
and set
\[
R_{ijk\ell}
=g(R(e_k,e_{\ell})e_i,e_j)=R_{ik\ell}^{\>\>j},
\quad \varOmega_{ij}:=\varOmega_{i}^{\>\>j},
\]
then we obtain
\[
\varOmega_{ij}
=\frac{1}{2}
\sum_{k,\ell=1}^{3} R_{ijk\ell}\,\theta^k\wedge\theta^{\ell}.
\]

\section{Contact manifolds}\label{sec:3}
As is well known contact structures play important roles in 
$3$-dimensional topology and geometry (see \textit{e.g.} \cite{Eliashberg, GeigesBook}). 
In this section we collect fundamental 
facts on contact structures in 
$3$-dimensional Riemannian geometry.
\subsection{}
Let $M$ be a $(2n+1)$-manifold. 
A $1$-form $\eta$ is called a 
\emph{contact form} on $M$ if 
$\eta\wedge (\mathrm{d}\eta)^{n}\not=0$. 
A $(2n+1)$-manifold
$M$ together with a contact form is called a 
\emph{contact manifold}. 
The distribution $\mathcal{D}$ defined by
\[
\mathcal{D}=\left \{ X \in \mathrm{T}M \ \vert \
\eta(X)=0
\right \}
\]
is called the 
\emph{contact structure} 
(or \emph{contact distribution})
determined by $\eta$. We denote  the 
vector subbundle of $\mathrm{T}M$ determined by 
$\mathcal{D}$ by the same letter $\mathcal{D}$:
\[
\mathcal{D}:=\bigcup_{p\in M}\mathcal{D}_{p},
\quad 
\mathcal{D}_{p}=\left \{ X_p \in \mathrm{T}_{p}M \ \vert \
\eta_{p}(X_{p})=0
\right \}.
\]
On a contact manifold $(M,\eta)$, 
there exists a unique vector
field $\xi$ such that
\[
\eta (\xi)=1,
\quad 
\mathrm{d} \eta (\xi,\cdot)=0.
\]
Namely $\xi$ is transversal to the contact structure $\mathcal{D}$. 
This vector field $\xi$ is called the \emph{Reeb vector field} 
(or \emph{characteriztic vector field}) of $(M,\eta)$.
\begin{definition}{\rm
Two contact manifolds 
$(M,\eta)$ and 
$(M^{\prime},\eta^{\prime})$ are said 
to be \emph{contactmorphic} if there exists 
a diffeomorphism (called a \emph{contactmorphism}) $f:M\to M^{\prime}$ which preserves the 
contact distribution, \textit{i.e.}, 
there exists a non-vanishing smooth function $\lambda$ 
on $M$ such that $f^{*}\eta^{\prime}=\lambda\,\eta$. 
When $M^{\prime}=M$, a contactmorphism $f$ is also called a 
\emph{contact transformation}. In particular, a contact transformation  
$f$ satisfying $f^{*}\eta=\eta$ is called a 
\emph{strict contact transformation}.
}
\end{definition}
On a contact manifold $(M,\eta)$, 
there exists an endomorphism field $\varphi$ and a
Riemannian metric $g$ on a contact manifold $(M,\eta)$ such that
\begin{equation}\label{almostcontact}
\varphi^2=-I+\eta \otimes \xi, \quad 
\eta(\xi)=1,
\end{equation}
\begin{equation}\label{associatedmetric}
g(\varphi X,\varphi Y)=g(X,Y)-\eta(X)\eta(Y),
\quad 
g(\xi,\cdot)=\eta,
\end{equation}
\begin{equation}\label{contactmetric}
\mathrm{d}\eta(X,Y)=g(X,\varphi Y)
\end{equation}
for all vector fields $X,\ Y$ on $M$. 
Here we use the convention
\[
\mathrm{d}\eta(X,Y)=\frac{1}{2}
\left(
X(\eta(Y))-Y(\eta(X))-\eta([X,Y])
\right).
\]
The pair $(\varphi,g)$ (or
quartet $(\eta,\xi,\varphi,g)\, $) is called the 
\emph{associated
almost contact Riemannian structure} of $(M,\eta)$. 
A contact manifold $(M,\eta)$ together with
an associated contact Riemannian structure 
is called a \emph{contact Riemannian manifold}
(or \emph{contact metric manifold}) 
and denote it by $M=(M,\eta,\xi,\varphi,g)$.

\begin{definition}
{\rm
A diffeomorphism $f:(M,\eta,\xi,\varphi, g)
\to (M^{\prime},\eta^{\prime},\xi^{\prime},\varphi^{\prime},g^{\prime})$ 
between contact Riemannian manifolds 
is called an \emph{isomorphism} if it preserves all structure tensors,
\textit{i.e}.,
\[
\mathrm{d}f\circ \varphi=\varphi^{\prime}\circ \mathrm{d}f,\ 
\quad 
\mathrm{d}f(\xi)=\xi^{\prime},
\quad 
f^{*}\eta^{\prime}=\eta,
\quad 
f^{*}g^{\prime}=g.
\]
}
\end{definition}
An isomorphism from a 
contact Riemannian manifold 
$(M,\eta,\xi,\varphi, g)$ to itself 
is called an \emph{automorphism}.
The set of all automorphisms 
$\mathrm{Aut}(M)$ is a subgroup 
of the isometry group $\mathrm{Iso}(M)$. 
Tanno showed the following fundamental fact.

\begin{proposition}[\cite{TannoDiff}]
\label{Tannoholomorphic}
If a diffeomorphism $f:M\to M^{\prime}$ between 
contact Riemannian manifolds is $\varphi$-holomorphic,
i.e., $\mathrm{d}f\circ \varphi=\varphi^{\prime}\circ \mathrm{d}f$, then
there exists a positive constant $a$ such that
\[
f^{*}\eta^{\prime}=a\eta,
\quad 
\mathrm{d}f(\xi)=a\xi^{\prime},
\quad 
f^{*}g^{\prime}=ag+a(a-1)\eta \otimes \eta.
\]
\end{proposition}
Tanno \cite{Ta0} showed that $\dim\mathrm{Aut}(M)\leq (n+1)^2$.
In particular $\dim\mathrm{Aut}(M)\leq 4$ for contact 
Riemannian $3$-manifolds.
\subsection{}
Next, on the direct product $M \times \mathbb{R}(t)$ of a contact Riemannian manifold 
$M$ and the real line $\mathbb{R}(t)$, we can extend
naturally the endomorphism field $\varphi$ to an almost complex
structure $J$ on $M\times \mathbb{R}(t)$:
\[
J\left(X,f\frac{\partial}{\partial t}\right)=
\left(\varphi X-f\xi, \eta(X)\frac{\partial}{\partial t}\right),\
X \in 
\varGamma(\mathrm{T}M),\>\> f \in
C^{\infty}(M \times \mathbb{R}).
\]
If the almost complex structure $J$ on $M \times \mathbb{R}$ is
integrable then $(M,\eta)$ is said to be \emph{normal}. The
normality is equivalent to the vanishing of the 
\emph{Sasaki-Hatakeyama torsion}
$\mathcal{N}$:
\[
\mathcal{N}(X,Y)=[\varphi,\varphi](X,Y)
+2\mathrm{d} \eta(X,Y).
\]
Here $[\varphi,\varphi]$ is the 
\emph{Nijenhuis torsion} of
$\varphi$:
\[
[\varphi,\varphi](X,Y)=
\varphi^2[X,Y]+[\varphi X,\varphi Y]
-\varphi[\varphi X,Y]
-\varphi[X,\varphi Y],\quad 
X,Y \in \varGamma(\mathrm{T}M).
\]
The product metric $g+\mathrm{d}t^2$ on $M \times \mathbb{R}$ is $J$-Hermitian. 
When $M$ is normal, then $(M\times\mathbb{R},J,g+\mathrm{d}t^2)$ is a locally 
conformal K{\"a}hler manifold. Topological properties 
of normal (almost) contact Riemannian $3$-manifolds, we refer to 
\cite{Geiges}.

The normality is reinterpreted as the integrability 
of the almost complex structure $J$ on 
the product manifold $\mathbb{R}^{+}\times M$ defined by
\[
JX=r\eta(X)\frac{\partial}{\partial r}+\varphi X,
\quad 
J\frac{\partial}{\partial r}=-\frac{1}{r}\xi,\quad X\in\varGamma(\mathrm{T}M),
\]
where $r>0$ is the radius coordinate. 
The vector field $r\partial/\partial_{r}$ is called the 
\emph{Liouville vector field}.

A normal contact Riemannian manifold
$(M,\eta,\xi,\varphi,g)$ is called a 
\emph{Sasakian manifold}
(or \emph{Sasaki manifold}). 
One can see that a Rimannian manifold $M$ is 
Sasakian if and only if its cone $C(M)=\mathbb{R}^{+}\times M$ equipped with 
the cone metric $\mathrm{d}r^2+r^2\,g$ is a K{\"a}hler manifold \cite{BG}. 

On a contact Riemannian manifold $M$, we introduce an 
endomorphism field $h$ by 
$h=(\pounds_{\xi}\varphi)/2$. Here $\pounds_{\xi}$ is the 
Lie differentiation by $\xi$. 
When $M$ is $3$-dimensional, we know the 
following characterization of normality.

\begin{proposition}
Let $(M,\eta,\xi,\varphi,g)$ be a contact Riemannian $3$-manifold.
Then the following five conditions are 
mutually equivalent.
\begin{itemize}
\item $M$ is a Sasakian manifold,
\item $\xi$ is a Killing vector field,
\item $\nabla\xi=-\varphi$,
\item $(\nabla_{X}\varphi)Y=g(X,Y)\xi-\eta(Y)X$
for any vector fields $X$ and $Y$ on $M$.
\item $h=0$ on $M$.
\end{itemize}
\end{proposition}
For more details on contact Riemannian manifolds, we refer to 
\cite{B,B2,BG}.

\subsection{The vector product}\label{sec7.1}
Let $(M,g,\mathrm{d}v_g)$ be an oriented Riemannian $3$-manifold 
Take a positively oriented 
local orthonormal frame field $\{e_1,e_2,e_3\}$. Then 
we have $\mathrm{d}v_{g}(e_1,e_2,e_3)=1$. 
Denote by $\{\theta^1,\theta^2,\theta^3\}$ the 
local orthonormal coframe field dual to 
$\{e_1,e_2,e_3\}$. Then 
we have 
\[
\mathrm{d}v_{g}=3!\,\theta^1\wedge\theta^2\wedge\theta^3.
\]
Let us denote by $\mathrm{d}V$ the tensor field of type $(1,2)$ associated 
to $\mathrm{d}v_g$, \textit{i.e.},
\[
\mathrm{d}v_{g}(X,Y,Z)=g(\mathrm{d}V(X,Y),Z).
\]
Equivalently, $\mathrm{d}v_g=(\mathrm{d}V)_{\flat}$.
The tensor field $\mathrm{d}V$ is interpreted as 
the distribution of the \emph{vector product operation} 
(also called the \emph{cross product}) 
$\times$ on each tangent space $\mathrm{T}_{p}M$. 
Indeed, $X\times Y$ is defined by 
\[
g(X\times Y,Z)=\mathrm{d}v_g(X,Y,Z),
\quad 
X,Y,Z \in \mathrm{T}_{p}M.
\]
Now let $(M,\varphi,\xi,\eta,g)$ be a contact Riemannian 
$3$-manifold. Then the volume element $dv_g$ derived from the
associated metric $g$ is related to the contact form $\eta$ by
\begin{equation}\label{volume}
\mathrm{d}v_{g}=-3\eta\wedge \mathrm{d}\eta.
\end{equation}
\begin{remark}{\rm 
In \cite[Theorem 4.6]{B2}, the volume element 
$\mathrm{d}v_g$ is expressed as
\[
\mathrm{d}v_{g}=-\frac{1}{2}\eta\wedge \mathrm{d}\eta.
\]
The convention for the volume element used in \cite{B2} is 
$\mathrm{d}v_g=\theta^1\wedge\theta^2\wedge\theta^3$. 
On the other hand, in the present article, we use the convention 
$\mathrm{d}v_g=3!\,\theta^1\wedge \theta^2\wedge\theta^3$ 
(see \cite[p.~67]{TV}). 
}
\end{remark}
With respect to the volume element 
\eqref{volume}, the vector product $\times$ is computed as
\begin{equation}\label{cross}
X \times Y=\mathrm{d}V(X,Y)=-\mathrm{d}\eta(X,Y)\xi+\eta(X)\varphi{Y}-\eta(Y)\varphi{X}.
\end{equation}
Note that for a unit vector field $X$ orthogonal to $\xi$, the 
local frame field $\{X, \varphi X, \xi\}$
is positively oriented and 
\[
\xi\times X=\varphi X. 
\]

\section{Some linear connections on contact Riemannian manifolds}
\subsection{Generalized Tanaka-Webster connection}
Tanno \cite{Tanno} introduced the following linear connection on 
contact Riemannian manifolds.
\begin{equation}
\nabla^{*}_{X}Y=\nabla_{X}Y
+\eta(X)\varphi{Y}+\{(\nabla_{X}\eta)Y\}\xi-\eta(Y)\nabla_{X}\xi.
\end{equation}
This connection is called the 
\emph{generalized Tanaka-Webster connection}.
The generalized Tanaka-Webster connection satisfies the following properties:
\[
\nabla^{*}\xi=0,\quad \nabla^{*}\eta=0,\quad \nabla^{*}g=0.
\]

\subsection{The associated CR-structure}
On a contact Riemannian manifold $M$ of arbitrary odd dimension, we 
consider a complex vector subbundle $\mathscr{S}$ of 
the complexified tangent bundle $T^{\mathbb C}M$ of $M$ defined by
\[
\mathscr{S}=\{X-\sqrt{-1}\varphi{X} \ \vert \ X\in \varGamma(\mathcal{D})\},
\]
where $\mathcal{D}$ is the contact distribution. 
We call $\mathscr{S}$ the \emph{almost CR-structure} associated to
$(\varphi,\xi,\eta,g)$. 
Define
a section $L$ of
$\varGamma(\mathcal{D}^*\otimes \mathcal{D}^*)$ by
\[
L(X,Y)=-\mathrm{d}\eta(X,\varphi{Y}).
\]
Then $L$ is positive definite on $\mathcal{D}\otimes\mathcal{D}$ and called the
\emph{Levi-form} of $M$.

An almost CR-structure is said to be \emph{integrable} if it 
satisfies 
\emph{integrability condition}:
\[
[\varGamma(\mathscr{S}),\varGamma(\mathscr{S})]\subset \varGamma(\mathscr{S}).
\]
When the associated almost CR-structure $\mathscr{S}$ is integrable,
the resulting CR-manifold $(M,\mathscr{S})$ is called a
\emph{strongly pseudo-convex CR-manifold}
or \emph{strongly pseudo-convex pseudo-Hermitian manifold}.
Note that the CR-integrability is equivalent to the vanishing 
of Sasaki-Hatakeyama torsion on the contact distribution. 
Thus the normality is stronger than CR-integrability. 
When $M$ is $3$-dimensional, then $\mathscr{S}$ is automatically integrable.
\begin{proposition}\label{TannoTheorem}
Let $M$ be a contact Riemannian manifold. Then its associated almost
CR-structure is integrable if and only if $\mathcal{Q}=0$. Here
the tensor field $\mathcal{Q}$ is defined by
\begin{equation}\label{Tanno-tensor}
\mathcal{Q}(Y,X):=(\nabla_{X}\varphi)Y+(\nabla_{X}\eta)(\varphi Y)\xi
+\eta(Y)\varphi(\nabla_{X}\xi).
\end{equation}
\end{proposition}
The tensor field
$\mathcal{Q}$ is called the \emph{Tanno tensor field}. Tanno 
tensor field is computed as
\[
\mathcal{Q}(Y,X)=(\nabla_{X}\varphi)Y
-g((I+h)X,Y)\xi
+\eta(Y)(I+h)X.
\]
Thus on a strongly pseudo-convex CR-manifolds, 
the following formula holds:
\begin{equation}\label{Tanno-formula}
(\nabla_{X}\varphi)Y=g((I+h)X,Y)\xi-\eta(Y)(I+h)X
\end{equation}
for all vector fields $X$ and $Y$.
The formula
\eqref{Tanno-formula} implies
\[
\nabla_{X}\xi=-\varphi(I+h)X, \quad
X \in \varGamma(\mathrm{T}M).
\]
According to Tanaka \cite{Tanaka}, 
a strongly pseudo-convex CR manifold is said to be 
\emph{normal} if its Reeb vector field is 
\emph{analytic}, that is, $[\xi,\varGamma(\mathscr{S})]\subset\varGamma(\mathscr{S})$.
The Reeb vector field is analytic if and only if $\xi$ is an infinitesimal 
contact transformation and $[X,\varphi Y]=\varphi[X,Y]$ for all 
$X$, $Y\in\varGamma(\mathcal{D})$. One can see that a strongly pseudo-convex CR-manifold 
is normal if and only if its underlying contact Riemannian structure is Sasakian.

\subsection{}
In the study of strongly pseudo-convex CR-manifolds,
the linear connection $\hat{\nabla}$ introduced by Tanaka and Webster
is highly useful:
\begin{equation}\label{Tanaka-Webster}
\hat{\nabla}_{X}Y=\nabla_{X}Y+\eta(X)\varphi{Y}
+\{(\nabla_{X}\eta)Y\}\xi-\eta(Y)\nabla_{X}\xi.
\end{equation}
Here $\nabla$ is the Levi-Civita connection of the
associated metric. The linear connection $\hat{\nabla}$
is referred as to the \emph{Tanaka-Webster connection} 
\cite{Tanaka,Web}. 
The Tanaka-Webster connection is rewritten as
\[
\hat{\nabla}_{X}Y=\nabla_{X}Y+\eta(X)\varphi{Y}
+\eta(Y)\varphi(I+h)X
-g(\varphi(I+h)X,Y)\xi.
\]
It should be remarked that the Tanaka-Webster connection has 
\textit{non-vanishing} torsion $\hat{T}$:
\[
\hat{T}(X,Y)=2g(X,\varphi{Y})\xi
+\eta(Y)\varphi{h}X-\eta(X)\varphi{h}Y.
\]
With respect to the Tanaka-Webster connection, 
all structure tensor fields
$(\varphi,\xi,\eta,g)$ are parallel,
\textit{i.e}.,
\[
\hat{\nabla}\varphi=0, \quad
\hat{\nabla}\xi=0, \quad
\hat{\nabla}\eta=0, \quad
\hat{\nabla}g=0.
\]
On a strongly pseudo-convex CR-manifold $M$, 
the generalized Tanaka-Webster connection $\nabla^{*}$ 
in the sense of Tanno coincides with the Tanaka-Webster connection 
$\hat{\nabla}$.

\section{Homogeneous contact Riemannian structures}\label{sec:5}
\subsection{}
Let $M$ be a contact Riemannian manifold.
Then $M$ is said to be a \emph{homogeneous contact Riemannian manifold}
if there exists a Lie group of isometries which preserves
 the contact form and acts transitively on $M$. 
Note that Proposition \ref{Tannoholomorphic} implies 
that every $\eta$-preserving isometry is an isomorphism.

Having in mind Sekigawa-Kri{\v c}henko theorem \cite{Sekigawa,Krichenko}
 (see also 
\cite{CalLo}) and Proposition \ref{Tannoholomorphic}, 
the notion of homogeneous contact Riemannian 
structure is defined as follows:
\begin{definition}{\rm 
Let $(M,\eta,\xi,\varphi,g)$ be an contact Riemannian manifold.
A \emph{homogeneous contact Riemannian structure} is a homogeneous 
Riemannian structure $S$ which satisfies the additional condition 
$\tilde{\nabla}\varphi=0$.
}
\end{definition}

Chinea and Gonzalez obtained the 
following fundamental result.

\begin{lemma}[\cite{CG}]
Let $S$ be a homogeneous contact Riemannian 
structure on an contact Riemannian manifold $M$.
Then the Ambrose-Singer connection $\tilde{\nabla}=\nabla+S$ satisfies
\[
\tilde{\nabla}\eta=0,\quad 
\tilde{\nabla}\xi=0.
\]
\end{lemma}
As a direct consequence of this Lemma, we have: 
\begin{theorem}[\cite{CG}]
Let $(M,\eta,\xi,\varphi,g)$ be a homogeneous 
contact Riemannian manifold.
Then there exists a homogeneous contact Riemannian 
structure $S$ on $M$.
Conversely, let $M$ be a simply connected and complete 
contact Riemannian manifold with a
homogeneous contact Riemannian structure $S$ 
then $M$ is a homogeneous contact Riemannian manifold.
\end{theorem}
For more information on homogeneous almost contact Riemannian 
structures, we refer to \cite{KW,Fino94,Fino,GC,WT}.

\section{Sasakian $\varphi$-symmetric spaces}\label{sec:6}
\subsection{}
Let $M$ be a Sasakian manifold. 
Okumura \cite{Ok} introduced the following one-parameter family of linear
connections on $M$:
\[
\nabla^{r}_{X}Y=\nabla_{X}Y+A^{r}(X)Y,
\]
where
\[
A^{r}(X)Y=\mathrm{d}\eta(X,Y)\xi-r\eta(X)\varphi{Y}+
\eta(Y)\varphi{X},\quad 
r \in \mathbb{R}.
\]
Note that the connection $\hat{\nabla}=:\nabla^{r}\vert_{r=-1}$
coincides with 
\emph{Tanaka-Webster connection}.
On the other hand, $\nabla^{r}\vert_{r=0}$ 
was introduced by Sasaki and Hatakeyama \cite{SH}.
The connection $\nabla^{r}\vert_{r=1}$ was introduced by Motomiya \cite{Mot} 
(see also \cite{KM}).
Some authors called $\nabla^r\vert_{r=1}$, the  
\emph{Okumura connection} (\cite{BV2}).

\begin{proposition}[\cite{Ok,Ta2}]
The linear connection $\nabla^{r}$ satisfeis
\[
\nabla^{r}g=0, 
\quad 
\nabla^{r}\varphi=0,
\quad 
\nabla^{r}\eta=0,
\quad 
\nabla^{r}\xi=0,
\quad 
\nabla^{r}A^{r}=0.
\]
\end{proposition}
This Proposition implies that if a Sasakian manifold $M$ satisfies 
$\nabla^{r}R=0$ then it is locally homogeneous Sasakian manifold with 
homogeneous contact Riemannian structure $A^{r}$.

\begin{proposition}
[\textit{cf.}~\cite{Inoguti,Ta2}]
The curvature $R^{r}$ and torsion $T^{r}$ of 
the connection ${\nabla}^{r}$
satisfies the following formulas:
\begin{align*}
T^{r}(X,Y)=&2\mathrm{d}\eta(X,Y)\xi-(r+1)\{\eta(X){\varphi}Y-
\eta(Y){\varphi}X\},\\
{R}^{r}(X,Y)Z=&R(X,Y)Z+\{\eta(Y)g(Z,X)-\eta(X)g(Y,Z)\}\xi\\
&+
\eta(Z)\eta(X)Y-\eta(Y)\eta(Z)X\\
&+
\mathrm{d}\eta(Y,Z){\varphi}X
+\mathrm{d}\eta(Z,X){\varphi}Y-2r\mathrm{d}\eta(X,Y){\varphi}Z,
\\
{R}^{r}(X,Y)\xi=&{R}^{r}(\xi,X)Y=0,\ \
\eta(R^{r}(X,Y)Z)=0,
\\
{R}^{r}({\varphi}X,{\varphi}Y)\varphi{Z}
=&{\varphi}{R}^{r}(X,Y)Z.
\end{align*}
for all $X$,$Y$,$Z \in \varGamma(\mathrm{T}M)$.
\end{proposition}

\begin{proposition}[\cite{Ta2}]
On a Sasakian manifold $M$, $\nabla^{r}R^{r}=0$ holds if and only if 
$\nabla^r{R}=0$.
\end{proposition}

\begin{definition}[\cite{BuV89,Ta2}]
{\rm
Let $M$ be a Sasakian manifold.
\begin{enumerate}
\item A geodesic $\gamma(s)$ parametrized by arclength in $M$ is said 
to be a $\varphi$-\emph{geodesic} if $\eta(\gamma^\prime)=0$.
\item A local diffeomorphism $s_{p}$ is said to be a 
$\varphi$-\emph{geodesic symmetry} with base point $p\in M$ if
for each $\varphi$-geodesic $\gamma(s)$ such that $\gamma(0)$ lies in the 
trajectory
of $\xi$ passing through $p$, $s_{p}\gamma(s)=\gamma(-s)$ for each $s$.
\item $M$ is said to be a \emph{locally $\varphi$-symmetric space} if
its $\varphi$-geodesic symmetries are \emph{isometric}.
\end{enumerate}
}
\end{definition}
Since the points of the Reeb flow through $p$ is are fixed
by $s_p$, one can see that $s_p$ is represented as
\[
s_{p}=\mathrm{exp}_{p}\circ (-I_{p}+2\eta_{p}\otimes \xi_{p})\circ 
\mathrm{exp}_{p}^{-1}.
\]
It should be remarked that every $\varphi$-geodesic symmetry 
is a local automorphism on a locally $\varphi$-symmetric space 
\cite{BuV89}. 
The local $\varphi$-symmetry is equivalent to 
CR-symmetry (see \cite{DL}).

Toshio Takahashi characterized locally $\varphi$-symmetric spaces in terms of 
Okumura connection as
follows:
\begin{proposition}[\cite{Ta2}]\label{prop:6.4}
A Sasakian manifold is locally $\varphi$-symmetric if and only if
the curvature tensor ${R}^{r}$ of the Okumura's connection 
$\nabla^{r}$ is 
parallel with respect to Okumura connection,
\textit{i.e.}, ${\nabla}^{r}{R}^{r}=0$.
\end{proposition}

Since the difference tensor $A^{1}$ satisfies 
${A}^{1}(X)X=0$, 
the following interesting result is deduced \cite{BV2}.
\begin{proposition}\label{prop:6.5}
Let $M$ be a 
locally $\varphi$-symmetric space.
Then ${A}^{1}$ defines a locally naturally reductive homogeneous
structure on $M$. 
\end{proposition} 
Here we specialize that $\dim M=3$. Take a 
local orthonormal frame field 
$\{e_1,e_2,e_3\}$ satisfying 
\[
\eta(e_1)=0,\quad e_2=\varphi e_1,\quad e_3=\xi.
\] 
Denote by $\{\theta^1,\theta^2,\theta^3\}$ the dual orthonormal
coframe field to $\{e_1,e_2,e_3\}$, then $A^r$ is locally expressed in 
the following form:
\begin{proposition}\label{prop:6.6}
On a Sasakian $3$-manifold $M$, the linear connection $\nabla^r$ is expressed as $\nabla^r=\nabla+A^r$ with
\[
A^{r}_{\flat}
=-2r\theta^3\otimes(\theta^1\wedge \theta^2)
-2\theta^1\otimes(\theta^2\wedge \theta^3)
-2\theta^2\otimes(\theta^3\wedge \theta^1).
\]
In particular $A^{1}=-\mathrm{d}V$. 
\end{proposition}
\begin{proof}
Since $A^r_{\flat}$ 
satisfies
\[
A^r_{\flat}(X,Y,Z)+A^r_{\flat}(X,Z,Y)=0,
\]
$A^{r}_{\flat}$ is 
expressed as
\[
A^{r}_{\flat}
=2A^{r}_{\flat}(X,e_1,e_2)(\theta^1\wedge \theta^2)
+2A^{r}_{\flat}(X,e_2,e_3)(\theta^2\wedge \theta^3)
+2A^{r}_{\flat}(X,e_3,e_1)(\theta^3\wedge \theta^1).
\]
By the definition of $A^r$, we get the required result.
\end{proof}
 \begin{corollary}\label{cor:6.1}
 Let $M$ be a $3$-dimensional Sasakian $\varphi$-symmetric space, 
 then $A^r$ is a homogeneous Riemannian structure of type 
 $\mathcal{T}_2\oplus\mathcal{T}_3$. Moreover $A^r$ is of type 
 $\mathcal{T}_2$ if and only if $r=-2$.
 \end{corollary}
 \begin{proof}
 Some calculation show that $c_{12}(A^r_{\flat})=0$ and 
\[ 
\underset{X,Y,Z}{\mathfrak{S}}A^{r}_{\flat}(X,Y,Z)
=-4(r+2)\{
\theta^3\otimes(\theta^1\wedge\theta^2)
+\theta^1\otimes(\theta^2\wedge\theta^3)
+\theta^2\otimes(\theta^3\wedge\theta^1)
\}.
 \]
 \end{proof}
For more information on Sasakian $\varphi$-symmetric spaces, 
we refer to \cite{BV3,BuV,JK,KoWe}.  
Recently Ohnita \cite{Ohnita21,Ohnita22} used the connection $\nabla^r$ with $r=-1/2$.

The following lemma can be verified by 
direct computation.
\begin{lemma}\label{lem:sigma-connection}
Let $M=(M,\eta,\xi,\varphi,g)$ be a Sasakian manifold and 
$\sigma$ be a $1$-form on $M$. 
Define a linear connection $\nabla^{\sigma}$ by 
\[
\nabla^{\sigma}_{X}Y=\nabla_{X}Y+\mathrm{d}\eta(X,Y)\xi
+\eta(Y)\varphi X+\sigma(X)\varphi Y,
\quad X,Y\in\varGamma(\mathrm{T}M)).
\]
Then $\nabla^{\sigma}$ satisfies
\[
\nabla^{\sigma}\varphi=0,
\quad 
\nabla^{\sigma}\xi=0,
\quad
\nabla^{\sigma}\eta=0,
\quad
\nabla^{\sigma}g=0.
\]
In particular, if we choose $\sigma=-r\,\eta$, then 
$\nabla^{\sigma}$ coincides with $\nabla^r$.
\end{lemma}

\section{Sasakian space forms}\label{sec:7}
\subsection{}
Let $(M,\eta,\xi,\varphi,g)$ be a contact Riemannian manifold 
(of arbitrary odd dimension). 
A tangent plane at a point of $M$ is said to be a 
\emph{holomorphic
plane} if it is invariant under $\varphi$. The sectional curvature
of a holomorphic plane is called \emph{holomorphic sectional
curvature}. If the sectional curvature function of $M$ is constant
on all holomorphic planes in $\mathrm{T}M$, then $M$ is said to be of 
\emph{constant holomorphic sectional curvature}.

It is known that 
every Sasakian manifold of constant holomorphic sectional curvature is 
locally $\varphi$-symmetric \cite{Ta2}. 
Conversely, every $3$-dimensional Sasakian $\varphi$-symmetric 
space is of constant holomorphic sectional 
curvature \cite{BV1}.

Complete and connected Sasakian manifolds of constant holomorphic
sectional curvature are called \emph{Sasakian space forms}.

\begin{example}[The unit sphere and Berger spheres]
{\rm Let $\mathbb{C}^{n+1}$ be the complex Euclidean 
$(n+1)$-space. Then the 
unit $(2n+1)$-sphere 
$\mathbb{S}^{2n+1}\subset\mathbb{C}^{n+1}$ inherits 
a Sasakian structure $(\varphi_1,\xi_1,\eta_1,g_1)$ from the standard K{\"a}hler structure. 
The resulting Sasakian manifold is a Sasakian space form of 
constant holomorphic sectional curvature $1$. 

Next, deform the Sasakian structure  $(\varphi_1,\xi_1,\eta_1,g_1)$ 
of $\mathbb{S}^{2n+1}$ as
\begin{equation}\label{eq:D-dorm}
\varphi_c:=\varphi_1,
\quad 
\eta_c:=\frac{4}{c+3}\eta_1,
\quad 
\xi_c:=\frac{c+3}{4}\xi_1,
\quad 
g_c:=\frac{4}{c+3}g_1-\frac{4(c-1)}{(c+3)^2}
\eta_1\otimes \eta_1,\quad c>-3.
\end{equation}
Then $(\mathbb{S}^{2n+1},\varphi_c,\xi_c,\eta_c,g_c)$ 
is a Sasakian space form of constant holomorphic sectional curvature 
$c>-3$. In this article the Sasakian space form 
$(\mathbb{S}^{2n+1},\varphi_c,\xi_c,\eta_c,g_c)$ is called the 
\emph{Berger sphere}. 
}
\end{example}

\begin{remark}{\rm Precisely speaking, 
the Berger sphere introduced by Berger 
is $(\mathbb{S}^3,\frac{c+3}{4}g_{c})$ with $c>1$.
}
\end{remark}

\begin{example}[Heisenberg group]
{\rm The \emph{Heisenberg group} 
$\mathrm{Nil}_3$ is the model space of 
nilgeometry in the sense of Thurston. 
The model space $\mathrm{Nil}_3$ is realized as
the Cartesian $3$-space $\mathbb{R}^3(x,y,z)$ 
together with the nilpotent Lie group structure
\[
(x,y,z)(x^{\prime},y^{\prime},z^{\prime})
=(x+x^{\prime},y+y^{\prime},
z+z^{\prime}+(xy^{\prime}-x^{\prime}y)/2)
\]
and the left invariant metric
\[
g=\frac{\mathrm{d}x^2+\mathrm{d}y^2}{4}
+\eta\otimes\eta,
\quad 
\eta=\frac{1}{2}\left(
\mathrm{d}z+\frac{y\,\mathrm{d}x-x\,\mathrm{d}y}{2}
\right).
\]
The $1$-form $\eta$ is a contact form with Reeb vector field $\xi=2\partial/
\partial{z}$. The contact form $\eta$ induces 
the endomorphism field $\varphi$ as 
\[
\varphi\frac{\partial}{\partial x}=
\frac{\partial}{\partial y}+\frac{x}{2}\frac{\partial}{\partial z},
\quad 
\varphi\frac{\partial}{\partial y}=
-\frac{\partial}{\partial y}+\frac{y}{2}\frac{\partial}{\partial z},
\quad 
\varphi\frac{\partial}{\partial z}=0.
\]
The Heisenberg group $\mathrm{Nil}_3$ equipped with the 
left invariant contact Riemannian structure $(\varphi,\xi,\eta,g)$ is a 
simply  
connected Sasakian space form of constant holomorphic sectional 
curvature $-3$.
}
\end{example}
\begin{example}\label{eg:hyperbolicSasakian} 
{\rm 
Let $\mathbb{H}^2(-\alpha^2)$ be the upper half plane 
equipped with the Poincar{\'e} metric
\[
\bar{g}=\frac{\mathrm{d}x^2+\mathrm{d}y^2}{\alpha^2y^2}
\]
of constant curvature $-\alpha^2<0$ as before.  
On the product manifold
\[
\mathbb{H}^2(-\alpha^2)\times\mathbb{R}
=\{(x,y,z)\in\mathbb{R}^3\>|\>y>0
\},
\]
we equip the 
following contact Riemannian structure:
\[
\eta=\mathrm{d}z-\frac{2\mathrm{d}x}{\alpha^2 y},
\quad \xi=\frac{\partial}{\partial z},
\quad 
g=\frac{\mathrm{d}x^2+\mathrm{d}y^2}{\alpha^2y^2}+\eta\otimes\eta,
\]
\[
\varphi \frac{\partial}{\partial x}
=\frac{\partial}{\partial y},
\quad 
\varphi \frac{\partial}{\partial y}
=-\frac{\partial}{\partial x}
-\frac{2}{\alpha^2 y}\frac{\partial}{\partial z},
\quad 
\varphi \frac{\partial}{\partial z}=0.
\]
Then the resulting contact Riemannian manifold 
$\mathscr{M}^3(c)=(\mathbb{H}^2(-\alpha^2)\times\mathbb{R},\eta,\xi,
\varphi,g)$ is a simply connected 
Sasakian space form of 
constant holomorphic sectional curvature $c=-3-\alpha^2<-3$. 
As a Riemannian manifold, $\mathscr{M}^3(c)=(\mathbb{H}^2(-\alpha^2)\times\mathbb{R},g)$ 
is isometric the the model space of SL-geometry in the sense of Thurston. 
The projection $\pi:\mathscr{M}^3(c)\to\mathbb{H}^2(c+3)$ is a 
Riemannian submersion and defines a 
principal line bundle over $\mathbb{H}^2(c+3)$.
}
\end{example}

Tanno \cite{Ta} classified simply connected Sasakian space forms. 
Moreover if a Sasakian $(2n+1)$-manifold has the automorphims group 
of maximum dimension $(n+1)^2$, then it is locally isomorphic to a Sasakian space 
form.
Tanno's classification is 
described as follows:

\begin{proposition}{\rm (\cite{Ta})}
Let $\mathscr{M}^{3}(c)$ be a 
$3$-dimensional simply connected Sasakian space
form of constant holomorphic sectional curvature $c$. 
Then $\mathscr{M}^{3}(c)$ is isomorphic to one of the following spaces{\rm:}
\begin{itemize} 
\item $c=1:$ The unit $3$-sphere $\mathbb{S}^3$ equipped with the standard Sasakian 
structure.
\item $c>1$ or $-3<c<1:$ The Berger $3$-sphere $(\mathbb{S}^3,\varphi_c,\xi_c,\eta_c,g_c)$ equipped with 
a Sasakian structure obtained by the deformation \eqref{eq:D-dorm} from the unit $3$-sphere $\mathbb{S}^3$.
\item $c=-3:$ The Heisenberg group $\mathrm{Nil}_3$.
\item $c<-3:$ The Sasakian space form $\mathbb{H}^{2}(c+3)\times\mathbb{R}$.
\end{itemize}
\end{proposition}

As we mentioned above, every $3$-dimensional 
Sasakian space form has $4$-dimensional automorphism group. 
Simply connected $3$-dimensional Sasakian spaces 
are expressed as homogeneous Sasakian $3$-manifolds with 
$4$-dimensional automorphism group as follows (\textit{cf}. \cite{BDI2,BV1}) :

\begin{proposition}
Let $\mathscr{M}^{3}(c)$ be a 
$3$-dimensional simply connected Sasakian space
form of constant holomorphic sectional curvature $c$. 
Then $\mathscr{M}^{3}(c)$ is isomorphic to one of the following 
homogeneous Sasakian manifolds of with $4$-dimensional 
automorphism group{\rm:}
\begin{itemize} 
\item $c>1:$ The Berger $3$-sphere equipped with a homogeneous Sasakian structure. 
The Berger $3$-sphere is realized as 
the special unitary group $\mathrm{SU}(2)$ equipped with a 
left invariant Sasakian structure. 
Moreover the Berger sphere can be represented by 
$\mathrm{U}(2)/\mathrm{U}(1)$ as a normal homogeneous space 
equipped with a homogeneous Sasakian structure. 
\item $c=1:$ The unit $3$-sphere $\mathbb{S}^3$ equipped with a 
homogeneous Sasakian structure. 
The unit $3$-sphere $\mathbb{S}^3$ is identified as $\mathrm{SU}(2)$ equipped with a 
left invariant Sasakian structure. Moreover 
$\mathbb{S}^3$ can be represented by $\mathbb{S}^3=\mathrm{U}(2)/
\mathrm{U}(1)$ as a normal homogeneous space 
equipped with a homogeneous Sasakian structure. 
\item $-3<c<1:$ 
The Berger $3$-sphere equipped with a homogeneous Sasakian structure. 
The Berger $3$-sphere is identified as $\mathrm{SU}(2)$ equipped with a 
left invariant Sasakian structure. Moreover the Berger $3$-sphere is represented by 
$(\mathrm{SU}(2)\times\mathrm{U}(1))/
\mathrm{U}(1)$ as a naturally reductive homogeneous space. 
The naturally reductive homogeneous space $(\mathrm{SU}(2)\times\mathrm{U}(1))/
\mathrm{U}(1)$ is non-normal. 
\item $c=-3:$ Heisenberg group
$\mathrm{Nil}_{3}$ equipped with a left invariant Sasakian structure.
The Heisenberg group can be represented by 
$(\mathrm{Nil}_3\ltimes\mathrm{U}(1))/\mathrm{U}(1)$ as a 
naturally reductive homogeneous equipped with a homogeneous Sasakian structure. 
\item $c<-3:$ The universal covering 
group $\widetilde{\mathrm{SL}}_2\mathbb{R}$ of the real special linear group 
$\mathrm{SL}_2\mathbb{R}$ equipped with a left invariant Sasakian structure. 
Those Sasakian space forms can be represented by 
$(\widetilde{\mathrm{SL}}_2\mathbb{R}\times\mathrm{SO}(2))/\mathrm{SO}(2)$ as a
naturally reductive homogeneous space equipped with a 
homogeneous Sasakian structure.
\end{itemize}
\end{proposition}
All $3$-dimensional Sasakian space forms 
belong to the so-called Bianchi-Cartan-Vranceanu family 
\cite{Bi,Ca,V} (see also \cite{BDI2}).

On the other hand, every simply connected $3$-dimensional 
Sasakian space form $\mathscr{M}^{3}(c)$ itself is a Lie group, 
thus $\mathscr{M}^{3}(c)$ can be represented as a coset manifold  
$\mathscr{M}^{3}(c)/\{\mathsf{e}\}$ where $\mathsf{e}$ is the unit element.
Thus at least, $\mathscr{M}^{3}(c)$ has two coset representations:
$\mathrm{Aut}(\mathscr{M}^{3}(c))/\mathrm{SO}(2)$ and $\mathscr{M}^{3}(c)/\{\mathsf{e}\}$.

The reductive decomposition 
corresponding to the naturally reductive homogeneous structure of 
the Berger $3$-sphere and $\mathrm{SL}_2\mathbb{R}$ are 
explicitly given in \cite{IM1} and \cite{IM2}, respectively.
\subsection{Sasakian Lie algebras}

\begin{definition}\cite[Definition 2]{AFV}
{\rm 
Let $\mathfrak{a}$ be Lie algebra 
of dimension $2n+1>1$.  
A quadruple 
$(\varphi,\xi,\eta,\langle\cdot,
\cdot\rangle$) on $\mathfrak{g}$ consisting 
of a linear endomorphism 
$\varphi\in\mathrm{End}
(\mathfrak{g})$, a vector 
$\xi\in\mathfrak{g}$, a covector 
$\eta\in\mathfrak{g}^{*}$ and an 
inner product 
$\langle\cdot,
\cdot\rangle$ on $\mathfrak{g}$ 
is said to be 
a \emph{Sasakian structure} 
if 
it satisfies 
\[
\varphi^2=-\mathrm{I}+\eta\otimes\xi, 
\quad \eta(\xi)=1,
\]
\[
\langle\varphi X,\varphi Y\rangle
=\langle X,Y\rangle-\eta(X)\eta(Y),
\]
\[
2\langle X,\varphi Y\rangle=-\eta([X,Y]),
\] 
\[
\varphi^2[X,Y]+[\varphi X,
\varphi Y]-\varphi[\varphi X,Y]
-\varphi[X,\varphi Y]
-\eta([X,Y])=0.
\]
A Lie algebra $\mathfrak{g}$ equipped with a 
Sasakian structure is called a 
\emph{Sasakian Lie algebra}.
}
\end{definition}
Let $G$ be a Lie group with Sasakian 
Lie algebra $\mathfrak{g}
=(\mathfrak{g},\varphi,\xi,\eta,\langle\cdot,
\cdot\rangle)$, then the Sasakian structure 
of $\mathfrak{g}$ induces a left invariant 
Sasakian structure on $G$. 

\begin{proposition}[\cite{AFV}]
The center $\mathfrak{z}(\mathfrak{g})$ of a Sasakian Lie 
algebra $\mathfrak{g}$ satisfies
\begin{itemize}
\item $\dim\mathfrak{z}(\mathfrak{g})\leq 1$.
\item If $\dim\mathfrak{z}(\mathfrak{g})=1$, then 
$\mathfrak{z}(\mathfrak{g})=\mathbb{R}\xi$.
\end{itemize}
\end{proposition}
According to Boothby and Wang 
\cite[Theorem 5]{BW}, 
the only real semi-simple Lie algebras 
admitting a contact form are 
$\mathfrak{su}(2)$ and $\mathfrak{sl}_2\mathbb{R}$.

On a Sasakian Lie algebra $\mathfrak{g}$, 
we obtain an orthogonal direct sum decomposition 
(see \cite[Corollary 7]{AFV}):
\[
\mathfrak{g}=\mathrm{Ker}\>\mathrm{ad}(\xi)
\oplus
\mathrm{Im}\,\mathrm{ad}(\xi).
\]
Moreover both $\mathrm{Ker}\>\mathrm{ad}(\xi)$ and 
$\mathrm{Im}\,\mathrm{ad}(\xi)$ are 
$\varphi$-invariant linear subspaces.

Now let $\mathfrak{g}$ be a 
$3$-dimensional Sasakian Lie algebra. 
If $[\mathfrak{g},\mathfrak{g}]=\mathfrak{g}$, 
then 
$\mathfrak{g}$ is semi-simple and hence 
$\mathfrak{g}\cong\mathfrak{su}(2)$ or 
$\mathfrak{g}\cong\mathfrak{sl}_2\mathbb{R}$.

Next, when $[\mathfrak{g},\mathfrak{g}]\not=\mathfrak{g}$, 
then $\mathfrak{g}$ is solvable.
One can confirm that 
$\dim (\mathrm{Ker}\,\mathrm{ad}(\xi)
\cap \mathrm{Ker}\,\eta)=2$ and 
hence $\xi$ is an element of the center 
$\mathfrak{z}(\mathfrak{g})$ of $\mathfrak{g}$. Note that 
$\dim  \mathfrak{z}(\mathfrak{g})=1$. 
The quotient Lie algebra 
$\mathfrak{g}/\mathfrak{z}(\mathfrak{g})$ is $2$-dimensional, 
thus it 
is isomorphic to 
the abelian Lie algebra $\mathbb{R}^2$ or the 
Lie algebra $\mathfrak{ga}(1)$ of the 
affine transformation group 
\[
\mathrm{GA}(1)=\mathrm{GL}_{1}\mathbb{R}\ltimes\mathbb{R}
=\left\{
\left.
\left(
\begin{array}{cc}
y & x\\
0 &1
\end{array}
\right)
\>
\right|
\>x,y\in\mathbb{R},\,y\not=0
\right\}
\]
of the 
real line. 
The Sasakian structure of $\mathfrak{g}$ 
induces a K{\"a}hler algebra structure on 
$\mathfrak{z}(\mathfrak{g})$. In the former case $\mathfrak{g}$ is 
isomorphic to the Heisenberg algebra 
$\mathfrak{nil}_3$.
In the latter case, $\mathfrak{g}$ is isomorphic to the 
direct sum $\mathfrak{ga}(1)\oplus\mathbb{R}$ 
of 
$\mathfrak{ga}(1)$ and the abelian Lie algebra $\mathbb{R}$.

There exists a basis 
$\{\theta^1,\theta^2,\theta^3\}$ of $\mathfrak{g}^{*}$ 
satisfying
\[
\mathrm{d}\theta^1=0,\quad 
\mathrm{d}\theta^2=\omega_1^2,\quad 
\mathrm{d}\theta^3=2\omega_1^2 
\]
\[
\xi=e_3,\quad E_1=e_1,\quad E_2=e_2,\quad E_3=\xi-2e_2
\]
Thus, all $3$-dimensional Sasakian algebras are classified as follows 
(see \cite{AFV}):
\begin{theorem}
Any $3$-dimensional Sasakian 
Lie algebra is isomorphic to 
one of the following ones{\rm:}
\[
\mathfrak{su}(2), 
\quad 
\mathfrak{sl}_2\mathbb{R},
\quad 
\mathfrak{nil}_3, 
\quad 
\mathfrak{ga}(1)\oplus\mathbb{R}.
\]
\end{theorem}
Here we exhibit an explicit model 
of the simply connected Lie group corresponding to 
$\mathfrak{ga}(1)\oplus\mathbb{R}$. Assume that $\tilde{G}(c)$ 
is a $3$-dimensional 
simply connected non-unimodular Lie group equipped with a left invariant Sasakian structure 
$(\eta,\varphi,\xi,g)$, then $\tilde{G}(c)$ is a Sasakian space form 
of constant holomorphic sectional curvature $c<-3$ (\cite{Perrone}, see also 
\cite{I-UU}). 
Set $c=-3-\alpha^2$ for some non-zero constant $\alpha$, then 
there exists an orthonormal basis 
$\{e_1,e_2,e_3\}$ of the Lie algebra $\mathfrak{g}$ of $G$ satisfying 
(see \textit{e.g.} \cite{Perrone}):
\begin{equation}\label{eq:nonuni}
[e_1,e_2]=\alpha e_2+2e_3,
\quad 
[e_2,e_3]=[e_3,e_1]=0.
\end{equation}
One can confirm that $\mathfrak{g}$ is isomorphic to 
$\mathfrak{ga}(1)\oplus\mathbb{R}$ (see \textit{e.g}, \cite{AFV}).

The left invariant Sasakian structure is
described as 
\begin{equation}\label{eq:nonuni-phi}
\varphi e_1=e_2,
\quad \varphi e_2=-e_1,
\quad \varphi e_3=0,
\quad \xi=e_3.
\end{equation}
The unimodular kenel 
\[
\mathfrak{u}=\{X\in\mathfrak{g}\>|\>\mathrm{tr}\>\mathrm{ad}(X)=0\}
\]
is spanned by $\{e_2,e_3\}$. One can see that 
$\tilde{G}(c)$ is isomorphic to the 
following solvable linear Lie group (\textit{cf.} \cite{IN,Milnor}):
\begin{equation}\label{eq:G}
\left\{
\left.
\left(
\begin{array}{cccc}
1 & 0 & 0 & x\\
0 & e^{\alpha x} & 0 & y\\
0 & \frac{2}{\alpha}(e^{\alpha x}-1) & 1 & z\\
0 & 0 & 0 &1
\end{array}
\right)
\>
\right|
\> x,y,z\in\mathbb{R}
\right\}.
\end{equation}
The orthonormal basis 
$\{e_1,e_2,e_3\}$ defines a left invariant vector fields:

\[
e_1=\frac{\partial}{\partial x},
\quad 
e_2=e^{\alpha x}\frac{\partial}{\partial y},
\quad 
e_3=\frac{2}{\alpha}(e^{\alpha x}-1)\frac{\partial}{\partial y}
+\frac{\partial}{\partial z}.
\]
Generally speaking, generic non-symmetric 
$3$-dimensional Riemannian Lie group $G=(G,\langle\cdot,\cdot\rangle)$  
(Lie group equipped with left invariant metrics) 
has unique expression. Here the uniqueness 
means that if a $3$-dimensional 
Riemannian Lie group 
$(G^{\prime},\langle\cdot,\cdot\rangle^{\prime})$ is 
isometric to $(G,\langle\cdot,\cdot\rangle)$ as a 
Riemannian $3$-manifold, 
then $G^{\prime}$ is 
isomorphic to $G$ as a Lie group. 
However as like $\widetilde{\mathrm{SL}}_2\mathbb{R}$ and 
$\tilde{G}(c)$, non-isomorphic $3$-dimensional 
Lie groups might admit left invariant metrics which make them isometric as Riemannian $3$-manifolds. In $3$-dimensional 
Lie group theory, other than 
$\widetilde{\mathrm{SL}}_2\mathbb{R}$ and 
$\tilde{G}(c)$, Euclidean $3$-space $\mathbb{E}^3$ and 
hyperbolic $3$-space $\mathbb{H}^3$ does not satisfy 
the uniqueness. For more detail, see \cite{I,MP}.

\subsection{}
Let us determine Ambrose-Singer connections 
on the simply connected non-unimodular Lie group 
$\tilde{G}(c)$ equipped with a left invariant Sasakian structure 
of constant holomorphic sectional curvature 
$c=-3-\alpha^2<-3$.

The Levi-Civita connection of $\tilde{G}(c)$ is given by the following table 
(\cite[p.~251]{Perrone}):
\begin{proposition}
Let $G$ be a Lie group equipped with a left 
invariant metric $g$. Assume that the Lie algebra $\mathfrak{g}$ of 
$G$ admits an orthonormal basis satisfying  
the commutation relations \eqref{eq:nonuni}, then 
the Levi-Civita connection $\nabla$ is described as
\[
\begin{array}{ccc}
\nabla_{e_1}e_{1}=0, & \nabla_{e_1}e_{2}=e_{3}, 
& 
\nabla_{e_1}e_{3}=-e_{2}\\
\nabla_{e_2}e_{1}=
-\alpha{e}_{2}-e_{3}, 
& \nabla_{e_2}e_{2}=\alpha e_1, & 
\nabla_{e_2}e_{3}=e_{1}\\
\nabla_{e_3}e_{1}=-e_{2},
& \nabla_{e_3}e_{2}= e_{1} &
 \nabla_{e_3}e_{3}=0.
\end{array}
\]
The Riemannian curvature $R$ is given by
\begin{align*}
R(e_1,e_2)e_1=&(
3+\alpha^2)e_{2},
\quad 
R(e_1,e_2)e_2=(-3-\alpha^2)
e_{1},
\quad 
R(e_1,e_2)e_3=0,
\\
R(e_1,e_3)e_1=& -e_{3},
\quad 
R(e_1,e_3)e_1=0,
\quad 
R(e_1,e_3)e_3=
e_{1},
\\
R(e_2,e_3)e_1=&0,
\quad 
R(e_2,e_3)e_2=
-e_{3},
\quad 
R(e_2,e_3)e_3=e_{2}.
\end{align*}
The sectional curvatures are
\[
K_{12}=-3-\alpha^2, 
\quad 
K_{13}=
K_{23}
=1.
\]
The non trivial components of the Ricci tensor field are
\[
R_{11}=
R_{22}=-\alpha^{2}-2,
\quad 
R_{33}=2.
\]
Hence the Ricci tensor field has the form{\rm:}
\[
\mathrm{Ric}=-(2+\alpha^2)g+(4+\alpha^2)\eta\otimes\eta,
\]
where $\eta$ is the metrical dual $1$-form of $\xi:=e_3$. 
Moreover the left invariant 
almost contact structure $(\eta,\xi,\varphi,g)$ 
defined by \eqref{eq:nonuni-phi} is Sasakian and of 
constant holomorphic sectional curvature $c=-3-\alpha^2$.
\end{proposition}
Let $\nabla^{(-)}$ be the Cartan-Schouten's 
$(-)$-connection of $\tilde{G}(c)$. Then 
$S=\nabla^{(-)}-\nabla$ is 
expressed by (see \cite[Theorem 1.3 (i.b)]{CFG}):
\begin{equation}\label{eq:minus}
S(X)Y=\mathrm{d}\eta(X,Y)\xi+\eta(Y)\varphi X+
(\sqrt{-(c+3)}\,\,\theta^2(X)+\eta(X))\varphi Y,
\quad X,Y\in\varGamma(\mathrm{T}G),
\end{equation}
where $\theta^2$ is the metrical dual $1$-form of $e_2$.
The homogeneous Riemannian structure 
$S$ is of type $\mathcal{T}_1\oplus
\mathcal{T}_2\oplus\mathcal{T}_3$ and can not be 
of type $\mathcal{T}_1$, $\mathcal{T}_2$ or $\mathcal{T}_3$.

The $(-)$-connection is different from 
Tanaka-Webster connection. Indeed, 
Tanaka-Webster 
connection $\hat{\nabla}$ 
is computed as follows.
\[
\hat{\nabla}_{e_1}e_1=
\hat{\nabla}_{e_1}e_2=
\hat{\nabla}_{e_1}e_{3}=0,
\]
\[
\hat{\nabla}_{e_2}e_1=-\alpha{e}_{2},
\quad 
\hat{\nabla}_{e_2}e_2=\alpha{e}_{1},
\quad
\hat{\nabla}_{e_2}e_{3}=0,
\]
\[
\hat{\nabla}_{e_3}e_1=
\hat{\nabla}_{e_3}e_2=
\hat{\nabla}_{e_3}e_{3}=0.
\]

\section{The homogeneous contact Riemannian structures on Sasakian space forms}
\label{sec:11}
In this section we carry out the classification of 
homogeneous contact Riemannian structures on $3$-dimensional Sasakian space forms.
\subsection{Connections and Curvatures}
Every simply connected and complete 
$3$-dimensional Sasakian space form $\mathscr{M}^{3}(c)$ of constant 
holomorphic sectional curvature is realized as a Lie group 
equipped with a left invariant contact Riemannian structure.
The Lie algebra of $\mathscr{M}^{3}(c)$ is spanned by 
an orthonormal basis $\{e_1,e_2,e_3\}$ satisfying 
the commutation relation (\textit{cf}. \cite{BDI2}):
\[
[e_1,e_2]=2e_3,\quad 
[e_2,e_3]=\frac{c+3}{2}e_1,
\quad
[e_3,e_1]=\frac{c+3}{2}e_2.
\]
Hence the Levi-Civita connection $\nabla$ is given 
by the following table: 
\[
\nabla_{e_1}e_1=0,
\quad
\nabla_{e_1}e_2=e_3,
\quad 
\nabla_{e_1}e_3=-e_2,
\]
\begin{equation}\label{eq:Levi-CivitaSS}
\nabla_{e_2}e_1=-e_3,
\quad 
\nabla_{e_2}e_2=0,
\quad 
\nabla_{e_2}e_3=e_1,
\end{equation}
\[
\nabla_{e_3}e_1=\frac{c+1}{2}e_2,
\quad 
\nabla_{e_3}e_2=-\frac{c+1}{2}e_1,
\quad 
\nabla_{e_3}e_3=0.
\]
The left invariant 
almost contact structure is defined by
\[
\varphi e_1=e_2,\quad \varphi e_2=-e_1, \quad 
\varphi e_3=0,
\quad \xi=e_3,\quad 
\eta=g(\cdot,e_3).
\]
Let 
$\Theta=(\theta^1,\theta^2,\theta^3)$ 
be the orthonormal coframe field 
metrically dual to $\{e_1,e_2,e_3\}$. 

From \eqref{eq:LC-omega} together with the table 
\eqref{eq:Levi-CivitaSS}, we obtain
\[
\omega_{2}^{\>\,1}=
-\frac{c+1}{2}
\theta^3,
\quad
\omega_{3}^{\>\,1}=\theta^2,
\quad 
\omega_{3}^{\>\,2}=-\theta^1.
\]
The Riemannian curvature $R$ of $\mathscr{M}^3(c)$ is described by
\[
R_{1212}=c,\quad  R_{1313}=R_{2323}=1.
\]
The sectional curvatures are
\[
K_{12}=c,\quad K_{13}=K_{23}=1.
\]
The Ricci tensor field
has components 
\[
R_{11}=R_{22}=c+1,\quad R_{33}=2
\]
and all the other components
are zero. Hence the scalar curvature is $2(c+2)$. 
Hence the Ricci tensor field is expressed as
\[
\mathrm{Ric}=(c+1)g+(1-c)\eta\otimes\eta.
\]

\subsection{The homogeneous structures}
In this section we prove the following main theorem 
of the present article. 
Let us represent $\mathscr{M}^{3}(c)$ as 
$\mathscr{M}^{3}(c)=(G\ltimes K)/K$, where
\[
G=\left.
\begin{cases}
\mathrm{SU}(2), & c>-3\\
\mathrm{Nil}_3, & c=-3,
\\
\widetilde{\mathrm{SL}}_2\mathbb{R}, & c<-3
\end{cases}
\right.
\quad 
K=\mathrm{SO}(2)\cong\mathrm{U}(1).
\]
Precisely speaking, when $c\not=-3$, 
$G\ltimes K$ is just a direct product $G\times K$.

Since the set of all homogeneous Riemannian structures
on the $3$-sphere $\mathbb{S}^{3}$ was determined by Abe \cite{Abe} 
(see section \ref{sec:Abe}), we concentrate our attention to Sasakian space forms with $c\not=1$. 

\begin{theorem}\label{thm:MainTheorem}
The set $\mathcal{S}$ of all homogeneous Riemannian structures
on a Sasakian space form $\mathscr{M}^3(c)=(G\ltimes K)/K$ with $c\not=1$ is 
described as follows{\rm:}
\begin{enumerate}
\item If $c\geq -3$, then  $\mathcal{S}=\{A^{r}\
\vert \ r\in \mathbb{R}\}$.  Namely the set of all Ambrose-Singer connections 
coincides with the one-parameter family of linear connections due to Okumura.
Moreover $\mathcal{S}$ coincides with the set of all homogeneous almost
contact Riemannian structures on $\mathscr{M}^3(c)$. 
The corresponding coset space representations are given as follows{\rm:}
\[
\mathscr{M}^3(c)=
\left\{
\begin{array}{cc}
(G\ltimes K)/K
& r\not=(c+1)/2
\\
G/\{\mathsf{e}\}
&
r=(c+1)/2.
\end{array}
\right.
\]
When $r=(c+1)/2$, the Ambrose-Singer 
connection $\nabla^r=\nabla+A^r$ is 
the Cartan-Schouten's $(-)$-connection of $G$. Every homogeneous Riemannian structure is of type 
$\mathcal{T}_2\oplus\mathcal{T}_3$. The homogeneous Riemannian structure $A^r$ is of type 
$\mathcal{T}_2$ if and only if $r=-2$.  
On the other hand, $A^r$ is of type $\mathcal{T}_3$ if and 
only if $r=1$. In this case 
$A^1=-\mathrm{d}V$.

\item If $c< -3$, then 
$\mathcal{S}=\{A^{r}\
\vert \ r\in \mathbb{R}\}\cup\{\nabla^{(-)}\}$, where 
$\nabla^{(-)}$ is the Cartan-Schouten's $(-)$-connection of 
the Sasakian Lie group $\mathrm{GA}^{+}(1)\times\mathbb{R}$. 
The set $\mathcal{S}$ coincides with the set of 
all homogeneous almost
contact Riemannian structures on $\mathscr{M}^3(c)$. 
The corresponding coset space representations are given as follows{\rm:}
\[
\mathscr{M}^3(c)=
\left\{
\begin{array}{ll}
(G\ltimes K)/K
& \nabla+A^{r},\quad r\not=(c+1)/2
\\
G/\{\mathsf{e}\}
&
\nabla+A^{r},\quad r=(c+1)/2
\\
(\mathrm{GA}^{+}(1)\times\mathbb{R})/\{\mathsf{e}\}
& \nabla^{(-)}.
\end{array}
\right.
\]
When $r=(c+1)/2$, the Ambrose-Singer 
connection $\nabla^r=\nabla+A^r$ is 
the Cartan-Schouten's $(-)$-connection of $G$. 
Every homogeneous Riemannian structure $A^{r}$ is of type 
$\mathcal{T}_2\oplus\mathcal{T}_3$. The homogeneous Riemannian structure $A^r$ is of type 
$\mathcal{T}_2$ if and only if $r=-2$.  
On the other hand, $A^r$ is of type $\mathcal{T}_3$ if and 
only if $r=1$. In this case 
$A^1=-\mathrm{d}V$. The Cartan-Schouten's $(-)$-connection for 
$\mathrm{GA}^{+}(1)\times\mathbb{R}$ is of type 
$\mathcal{T}_1\oplus\mathcal{T}_2\oplus\mathcal{T}_3$. 
\end{enumerate}
\end{theorem}

\subsection{Proof of Theorem \ref{thm:MainTheorem}}

Let $S_{\flat}$ be a tensor field of 
type $(0,3)$ satisfying
\[
S_{\flat}(X,Y,Z)+S_{\flat}(X,Z,Y)=0.
\]
Then 
$S_{\flat}$ is expressed as
\[
S_{\flat}(X,Y,Z)
=2S_{\flat}(X,e_1,e_2)(\theta^1\wedge \theta^2)
+2S_{\flat}(X,e_2,e_3)(\theta^2\wedge \theta^3)
+2S_{\flat}(X,e_3,e_1)(\theta^3\wedge \theta^1).
\]
We define a tensor field $S$ by
\[
S_{\flat}(X,Y,Z)=g(S(X)Y,Z).
\]
By using $S$, we define a linear connection
$\tilde{\nabla}$ by $\tilde{\nabla}=\nabla+S$. 
Assume that $S$ is a homogeneous Riemannian structure, then 
$\tilde{\nabla}$ satisfies 
$\tilde{\nabla}R=0$. 
Since $\mathscr{M}^{3}(c)$ is $3$-dimensional, 
this condition is 
equivalent to $\tilde{\nabla}\mathrm{Ric}=0$, 
\textit{i.e.},
\[
(\nabla_{X}\mathrm{Ric})(Y,Z)=\mathrm{Ric}(S(X)Y,Z)+\mathrm{Ric}(Y,S(X)Z).
\] 
From $\mathrm{Ric}=(c+1)g+(1-c)\eta\otimes\eta$, we get
\[
(\nabla_{X}\mathrm{Ric})
(e_i,e_j)
=-(1-c)\left(\omega_{i}^{\>3}(X)\eta(e_j)
+\omega_{j}^{\>3}(X)\eta(e_i)
\right).
\]
Next 
\[
\mathrm{Ric}(S(X)e_i,e_j)+\mathrm{Ric}(e_i,S(X)e_j)
=(1-c)\left(S(X,e_i,e_3)\eta(e_j)
+S(X,e_j,e_3)\eta(e_i)
\right).
\]
Hence we get
\[
\omega_{i}^{\>3}(X)\eta(e_j)
+\omega_{j}^{\>3}(X)\eta(e_i)
=
-S_{\flat}(X,e_i,e_3)\eta(e_j)
-S_{\flat}(X,e_j,e_3)\eta(e_i).
\]
From this equation we deduce that 
\[
S_{\flat}(X,e_3,e_1)=\omega_{1}^{\>\>3}(X)=-\theta^{2}(X),
\quad 
S_{\flat}(X,e_2,e_3)=-\omega_{2}^{3}(X)=-\theta^{1}(X).
\]
To determine $S_{\flat}(X,e_1,e_2)$, we compute
$\tilde{\nabla}S$. The covariant form $S_{\flat}$ of $S$ is 
expressed as
\[S_{\flat}
=2\sigma\otimes(\theta^1\wedge \theta^2)
-2\theta^1\otimes(\theta^2\wedge \theta^3)
-2\theta^2\otimes(\theta^3\wedge \theta^1),
\]
where the $1$-form $\sigma$ is defined by 
$\sigma(X)=S_{\flat}(X,e_1,e_2)$. 
The connection $1$-forms 
$\{\tilde{\omega}_{i}^{\>\>j}\}$ of $\tilde{\nabla}$ relative to 
$\{e_1,e_2,e_3\}$ are given by
\[
\tilde{\omega}_{i}^{\>\>j}(X)=g(\tilde{\nabla}_{X}e_i,e_j)
=\omega_{i}^{\>\>j}(X)+S(X,e_i,e_j).
\]
Hence 
\[
\tilde{\omega}_{1}^{\>\>2}(X)
=\frac{c+1}{2}\theta^3(X)+\sigma(X),
\quad
\tilde{\omega}_{1}^{\>\>3}(X)
=
\tilde{\omega}_{2}^{\>\>3}(X)=0,
\]
\[
\tilde{\nabla}_{X}\theta^1=
\left(\frac{c+1}{2}\theta^3(X)+\sigma(X)
\right)\theta^2,
\quad 
\tilde{\nabla}_{X}\theta^2=-
\left(\frac{c+1}{2}\theta^3(X)+\sigma(X)
\right)\theta^1,
\quad 
\tilde{\nabla}_{X}\theta^3=0.
\]
By using these 
\[
\tilde{\nabla}_{X}S_{\flat}=2(\tilde{\nabla}_{X}\sigma)(\theta^1\wedge \theta^2).
\]
Hence $\tilde{\nabla}S_{\flat}=0$ if and only if 
$\tilde{\nabla}\sigma=0$. 

Represent $\sigma$ as $\displaystyle\sigma=\sum_{k=1}^{3}\sigma_k\theta^k$, then 
$\tilde{\nabla}_{X}\sigma$ is computed as
\begin{align*}
\tilde{\nabla}_{X}\sigma=&\quad
\left\{(\mathrm{d}\sigma_1)(X)-
\left(\frac{c+1}{2}\theta^3(X)+\sigma(X)\right)\sigma_2
\right\}\theta^1
\\
&+
\left\{(\mathrm{d}\sigma_2)(X)+
\left(\frac{c+1}{2}\theta^3(X)+\sigma(X)\right)
\sigma_1
\right\}\theta^2
\\
&+(\mathrm{d}\sigma_3)(X)\,\theta^3.
\end{align*}
Thus the parallelism of $S$ with respect to $\tilde{\nabla}$ is 
equivalent to the differential system:
\[
\mathrm{d}\sigma_1=\sigma_{2}
\left\{\tfrac{c+1}{2}\theta^3+\sigma\right\},
\quad 
\mathrm{d}\sigma_2=-\sigma_{1}\left\{\tfrac{c+1}{2}\theta^3+\sigma\right\},
\quad 
\mathrm{d}\sigma_3=0.
\]
Hence $\sigma_3$ is a constant.

Let us investigate the solutions 
$\{\sigma_1,\sigma_2\}$ to this differential system.
\begin{lemma}\label{lem:sigma}
Let $\sigma_3$ be a constant. Then 
the solutions $\sigma_1,\sigma_2\in
 C^{\infty}(\mathscr{M}^3(c))$ to the differential system{\rm:}
 \[
\mathrm{d}\sigma_1=\sigma_{2}
\left(\sum_{k=1}^{3}\sigma_k\theta^k
+\frac{c+1}{2}\theta^3
\right),
\quad 
\mathrm{d}\sigma_2=-\sigma_{1}
\left(\sum_{k=1}^{3}\sigma_k\theta^k
+\frac{c+1}{2}\theta^3
\right)
 \]
 are given as follows{\rm:}
 \begin{enumerate}
 \item If {\rm (i)} $\sigma_3\not=1$ or {\rm (ii)} $\sigma_3=1$ and $c+3\geq 0$, then 
$\sigma_1$ and $\sigma_2$ are identically $0$ on $\mathscr{M}^3(c)$.
\item If $\sigma_3=1$ and $c+3<0$, then the solutions 
are 
\begin{itemize}
\item $\sigma_1$ and $\sigma_2$ are identically $0$ on $\mathscr{M}^3(c)$ or
\item $\sigma_1=\sqrt{-(c+3)}\cos\phi$ and 
$\sigma_2=\sqrt{-(c+3)}\sin\phi$ for some function $\phi$ satisfying 
\[
\mathrm{d}\phi=
-\left(
(\sqrt{-(c+3)}\cos \phi)\,\theta^1
+(\sqrt{-(c+3)}\cos \phi)\,\theta^2
+\frac{c+3}{2}\theta^3
\right).
\] 
\end{itemize}
\end{enumerate}
\end{lemma}
\begin{proof}
First of all, we remark that  the pair 
$\{\sigma_1=0, \sigma_2=0\}$ (both are identically zero) 
is a solution to the system. 
Here we assume that 
both of $\sigma_1$ and $\sigma_2$ are \emph{not} identically zero.

Let us consider the integrability condition 
of the system. 
The integrability condition 
is 
\[
(XY-YX-[X,Y])\sigma_{i}=0,\quad i=1,2.
\]
for any $X$, $Y\in\mathfrak{g}$. 
Equivalently, the integrability condition 
is 
\[
\mathrm{d}(\mathrm{d}\sigma_1)=\mathrm{d}(\mathrm{d}\sigma_2)=0.
\]
By straightforward computation we can deduce that 
the integrability condition is
\[
\sigma_{1}\,\mathrm{d}
\left(\sigma
+\frac{c+1}{2}\theta^3
\right)=0,
\quad 
\sigma_{2}\,\mathrm{d}
\left(\sigma
+\frac{c+1}{2}\theta^3
\right)=0,
\]
where $\sigma=\sigma_1\theta^1+\sigma_2\theta^2+\sigma_3\theta^3$. 
By the assumption we have
\[
\mathrm{d}
\left(\sigma
+\frac{c+1}{2}\theta^3
\right)=0.
\]
Hence there exists a smooth function $f$ such that 
\[
\sigma
+\frac{c+1}{2}\theta^3=\mathrm{d}f.
\]
Then we get
\[
e_{1}(f)=\sigma_1,
\quad 
e_{2}(f)=\sigma_2,
\quad 
e_{3}(f)=\sigma_3+\frac{c+1}{2}.
\]
From these we get
\[
[e_1,e_2]f=e_{1}(e_{2}(f))-e_{2}(e_{1}(f))=-\{(e_{1}(f))^2+(e_{2}(f))^2\}.
\]
On the other hand, from $[e_1,e_2]=2e_3$, we have
\[
[e_1,e_2]f=2e_{3}(f)=c+1+2\sigma_3.
\]
Thus we get
\begin{equation}\label{eq:8.2}
(e_{1}(f))^2+(e_{2}(f))^2=
-(c+1+2\sigma_3).
\end{equation}
Next, by using $[e_2,e_3]=(c+3)e_1/2$, we get 
\begin{equation}\label{eq:8.3}
\sigma_{1}(\sigma_3-1)=0.
\end{equation}
By using $[e_3,e_1]=(c+3)e_2/2$, we get 
\begin{equation}\label{eq:8.4}
\sigma_{2}(\sigma_3-1)=0. 
\end{equation}
If $\sigma_3\not=1$, then 
from \eqref{eq:8.3} and \eqref{eq:8.4}, 
we have $\sigma_1=\sigma_2=0$ (both are identically zero).  
This is a contradiction. Thus we have $\sigma_3=1$.

\begin{enumerate}
\item If $c+3\geq 0$, then from \eqref{eq:8.2}, 
\[
(e_{1}(f))^2+(e_{2}(f))^2=
-(c+3)\leq 0.
\]
Hence $\sigma_1=e_{1}(f)=0$ and  $\sigma_2=e_{2}(f)=0$. 
We know that even if $\sigma_3\not=1$, we have $\sigma_1=\sigma_2=0$. 
Thus the only solution of the system is $\sigma_1=\sigma_2=0$ when 
$c+3\geq 0$.
\item  If $c+3< 0$, then from \eqref{eq:8.2}, we have 
\[
\sigma_1^2+\sigma_2^2=-(c+3)>0.
\]
Hence $\sigma_1$ and $\sigma_2$ are expressed as
\begin{equation}\label{eq:8.5}
\sigma_1=\sqrt{-(c+3)}\,\cos\phi,
\quad 
\sigma_2=\sqrt{-(c+3)}\,\sin\phi
\end{equation}
for some function $\phi$. 
Inserting these formulas into the system, 
we obtain
\begin{equation}\label{eq:8.6}
\mathrm{d}\phi=-\left(
\sqrt{-(c+3)}\,\cos\phi\,\theta^1
+\sqrt{-(c+3)}\,\sin\phi\,\theta^2+\frac{c+3}{2}\,\theta^3
\right).
\end{equation}
Conversely let $\phi$ be a solution to \eqref{eq:8.6}. 
Define two functions $\sigma_1$ and $\sigma_2$ by 
\eqref{eq:8.5}, then $\{\sigma_1,\sigma_2\}$ 
is the solution of the system.
\end{enumerate}
\end{proof}
Let us continue to prove Theorem \ref{thm:MainTheorem}.
 
\subsubsection{The case {\rm (i)} $c+3\geq 0$ or {\rm (ii)} $c+3<0$ and $\sigma_3\not=1$}
In case $c+3\geq 0$, Lemma \ref{lem:sigma} implies 
that 
$\sigma_1=\sigma_2=0$. 
Next, in case $c+3<0$ and $\sigma_3\not=1$ we obtain 
$\sigma_1=\sigma_2=0$ again from Lemma \ref{lem:sigma}.

Hence $\sigma=\sigma_3\theta^3$ for some 
constant $\sigma_3$. Thus if we put $\sigma=-r\theta^3$ ($r\in\mathbb{R}$),
then we get 
\[
S_{\flat}=-2r \theta^{3}\otimes (\theta^{1}\wedge\theta^{2})
-2\theta^{1}\otimes (\theta^{2}\wedge\theta^{3})
-2\theta^{2}\otimes (\theta^{3}\wedge\theta^{1}),
\quad 
r \in \mathbb{R}.
\]
Obviously $S_{\flat}$ coincides with $A^{r}_{\flat}$. 
Thus the set of all Ambrose-Singer connections coinsides with the 
one-parameter family of Okumura's. 

The connection form $\tilde{\omega}$ and curvature form 
$\tilde{\varOmega}$ of the Ambrose-Singer connection 
$\tilde{\nabla}=\nabla+A^{r}$ are given by
\[
\tilde{\omega}=\left(r-\frac{c+1}{2}\right)
\left(
\begin{array}{ccc}
0 & \eta & 0\\
-\eta & 0 & 0\\
0 & 0 & 0
\end{array}
\right),
\quad 
\tilde{\varOmega}=
\left(r-\frac{c+1}{2}\right)
\left(
\begin{array}{ccc}
0 & -2\theta^1\wedge \theta^2 & 0\\
2\theta^1\wedge \theta^2 & 0 & 0\\
0 & 0 & 0
\end{array}
\right).
\]
Hence the isotropy subalgebra 
is trivial when and only when 
$r=(c+1)/2$. In this case we have 
$\tilde{\nabla}_{e_i}e_{j}=0$ for any 
$e_i$ and $e_j$. Thus the Ambrose-Singer 
connection is 
the Cartan-Schouten's $(-)$-connection.

\subsubsection{The case $c+3< 0$ and $\sigma_3=1$}
Next we consider the case $c+3<0$ and $\sigma_3=1$. 
In this case 

\begin{align*}
    \sigma_1&=\sqrt{-(c+3)} \cos \phi, \quad 
    \sigma_2=\sqrt{-(c+3)} \sin \phi, \quad 
    \sigma_3=1, \\
    e_1(\sigma_1)&=\sigma_1 \sigma_2, \quad 
    e_1(\sigma_2)=-(\sigma_1)^2, \\
    e_2(\sigma_1)&=(\sigma_2)^2, \quad 
    e_2(\sigma_2)=-\sigma_1 \sigma_2, \\
    e_3(\sigma_1)&=\frac{c+3}{2}\sigma_2, \quad
    e_3(\sigma_2)=-\frac{c+3}{2}\sigma_1.
\end{align*}
The components of $S$ are given by
\begin{align*}
S_{e_1}e_1&=\sigma_1 e_2, & S_{e_1}e_2&=-\sigma_1 e_1-e_3, & S_{e_1}e_3&=e_2, \\
S_{e_2}e_1&=\sigma_2 e_2+e_3, & S_{e_2}e_2&=-\sigma_2 e_1, & S_{e_2}e_3&=-e_1, \\
S_{e_3}e_1&=e_2, & S_{e_3}e_2&=-e_1, & S_{e_3}e_3&=0.
\end{align*}
By using the Sasakian structure, 
$S$ is expressed as
\[
S(X)Y=\mathrm{d}\eta(X,Y)\xi+\eta(Y)\varphi X+\sigma(X)\varphi Y.
\]
The Ambrose-Singer connection $\tilde{\nabla}=\nabla + S$ is described as
\begin{align*}
    \tilde{\nabla}_{e_1}e_1&=\sigma_1 e_2, & \tilde{\nabla}_{e_1}e_2&=-\sigma_1 e_1, & \tilde{\nabla}_{e_1}e_3&=0, \\
    \tilde{\nabla}_{e_2}e_1&=\sigma_2 e_2, & \tilde{\nabla}_{e_2}e_2&=-\sigma_2 e_1, & \tilde{\nabla}_{e_2}e_3&=0, \\
    \tilde{\nabla}_{e_3}e_1&=\frac{c+3}{2}e_2, & \tilde{\nabla}_{e_3}e_2&=-\frac{c+3}{2}e_1, & \tilde{\nabla}_{e_3}e_3&=0.
\end{align*}
One can confirm that the curvature tensor field $\tilde{R}$ of $\tilde{\nabla}$ vanishes.
For instance 
\begin{align*}
 \tilde{R} (e_1,e_2)e_1&=\tilde{\nabla}_{e_1}\tilde{\nabla}_{e_2}
 e_1-\tilde{\nabla}_{e_2}\tilde{\nabla}_{e_1}e_1
 -\tilde{\nabla}_{[e_1,e_2]}e_1 \\
    &=\tilde{\nabla}_{e_1}(\sigma_2e_2)-\tilde{\nabla}_{e_2}(\sigma_1e_2)-2\tilde{\nabla}_{e_3}e_1 \\
    &=(e_1\sigma_2)e_2+\sigma_2(-\sigma_1 e_1)-(e_2\sigma_1)e_2-\sigma_1
	(-\sigma_2e_1)-(c+3)e_2\\
    &=-(\sigma_1)^2e_2-(\sigma_2)^2e_2-(c+3)e_2=0.
\end{align*}
Moreover one can check that $\tilde{\nabla}\varphi=0$. 
Hence $S$ is a homogeneous contact Riemannian structure.

Let us change the left invariant orthonormal frame field
$\{e_1,e_2,e_3\}$ to 
\[
\tilde{e}_1=\frac{1}{\sqrt{-(c+3)}}
(\sigma_2 e_1-\sigma_1 e_2),
\quad 
\tilde{e}_2=\frac{1}{\sqrt{-(c+3)}}
(\sigma_1 e_1+\sigma_2e_2),
\quad 
\tilde{e}_3=e_3.
\]
Note that $\sqrt{\sigma_1^2+\sigma_2^2}
=\sqrt{-(c+3)}$ and 
\[
\varphi \tilde{e}_1=\tilde{e}_2,
\quad 
\varphi \tilde{e}_2=-\tilde{e}_1,
\quad 
\varphi \tilde{e}_3=0.
\]
It should be remarked that $\{\tilde{e}_1,\tilde{e}_2,\tilde{e}_3\}$ is 
\emph{not} left invariant. One can check that 
\[
\tilde{\nabla}_{\tilde{e}_i}\tilde{e}_j=0,\quad i,j=1,2,3.
\]

Let us determine the Lie group $L$ acting isometrically and 
transitively on $\mathscr{M}^3(c)$ corresponding to the 
homogeneous Riemannian structure $S$ and its 
reductive decomposition $\mathfrak{l}=\tilde{\mathfrak{h}}\oplus
\mathfrak{m}$. Here $\mathfrak{m}$ is identified with the 
tangent space of $\mathscr{M}^3(c)$ at the 
origin $o$.
Let us choose $o=\mathsf{e}$, the identity 
element of 
$\mathscr{M}^3(c)$, then $\mathfrak{m}=\mathfrak{g}$ 
as a real vector space equipped with an inner product. 
Since the isotropy algebra $\tilde{\mathfrak{h}}$ is 
the holonomy algebra spanned by the curvature operator $\tilde{R}$ 
of $\tilde{\nabla}$, we have $\tilde{\mathfrak{h}}=\{0\}$. 
Thus we have $\mathfrak{l}=\widetilde{\mathfrak h}\oplus
\mathfrak{m}=\{0\}\oplus\mathfrak{g}=\mathfrak{g}$. 
The basis 
$\{e_1, e_2, e_3\}$ is still an orthonormal basis of $\mathfrak{l}=\mathfrak{g}$.
The Lie blacket $[\cdot,\cdot]_{\mathfrak l}$ of $\mathfrak{l}$ is defined by
\begin{align*}
    [e_1,e_2]_{\mathfrak l}&=(-S_{e_1}e_2+S_{e_2}e_1)\vert_{\mathsf{e}} =\sigma_1(\mathsf{e})e_1+\sigma_2(\mathsf{e})e_2+2e_3, \\
    [e_2,e_3]_{\mathfrak l}&=(-S_{e_2}e_3+S_{e_3}e_2)\vert_{\mathsf{e}} =0, \\
    [e_3,e_1]_{\mathfrak l}&=(-S_{e_3}e_1+S_{e_1}e_3)\vert_{\mathsf{e}} 
    =0.
\end{align*}

Then we obtain
\begin{align*}
[\tilde{e}_1,\tilde{e}_2]_{\mathfrak l}
=&[e_1,e_2]_{\mathfrak l}
=\sigma_1(\mathsf{e}) e_1+\sigma_2(\mathsf{e}) e_2+2\tilde{e}_3
=\sqrt{-(c+3)}\tilde{e}_2+2\tilde{e}_3,
\\
[\tilde{e}_2,\tilde{e}_3]_{\mathfrak l}
=&
\frac{1}{\sqrt{-(c+3)}}\,[\sigma_1(\mathsf{e})e_1+\sigma_2(\mathsf{e})e_2,e_3]_{\mathfrak l}
=0,
\\
[\tilde{e}_3,\tilde{e}_1]_{\mathfrak l}
=&
\frac{1}{\sqrt{-(c+3)}}\,[
\sigma_2(\mathsf{e})e_1-\sigma_1(\mathsf{e})e_2,e_3]_{\mathfrak l}
=0.
\end{align*}
Hence the new orthonormal basis 
$\{\widetilde{e}_1,\widetilde{e}_2,\widetilde{e}_3\}$ 
satisfies the commutation relations \eqref{eq:nonuni} of 
$\mathfrak{ga}(1)\oplus\mathbb{R}$. 
Thus $L$ is isomorphic to 
the non-unimodular Lie group 
$\tilde{G}(c)$ and hence we obtain the 
coset space representation $\mathscr{M}^{3}(c)=\tilde{G}(c)
/\{\mathsf{e}\}$. 
Moreover 
$\{\widetilde{e}_1,\widetilde{e}_2,\widetilde{e}_3\}$ is left 
invariant with respect to the Lie group structure 
of $L$. Since $\tilde{\nabla}_{\tilde{e}_i}\tilde{e}_j=0$ for any 
$i,j=1,2,3$, the Ambrose-Singer connection 
$\tilde{\nabla}$ is the Cartan-Schouten's $(-)$-connection 
with respect to the Lie group structure 
$\tilde{G}(c)$. 

We interpret the $1$-form $\sigma$ in terms of 
$\mathfrak{l}=\mathfrak{ga}(1)\oplus\mathbb{R}$. 
Since $\tilde{\nabla}$ is 
the $(-)$-connection, 
the homogeneous Riemannian structure $S$ is 
left invariant with respect to the Lie group structure 
$\tilde{G}(c)$. 
So it suffices to compute the value of $\sigma$ at 
the origin $\mathsf{e}$. We have 
\[
\sigma(\tilde{e}_1)\vert_{\mathsf{e}}=0,
\quad
\sigma(\tilde{e}_2)\vert_{\mathsf{e}}=\sqrt{-(c+3)},
\quad 
\sigma(\tilde{e}_3)\vert_{\mathsf{e}}=1.
\]
Hence we get
\[
\sigma=\sqrt{-(c+3)}\,\,\tilde{\theta}^2+\eta,
\]
where $\tilde{\theta}^2$ is the metrical dual 
of $\tilde{e}_2$. Thus $S$ has the form:
\[
S(X)Y=\mathrm{d}\eta(X,Y)\xi+\eta(Y)\varphi X+
(\sqrt{-(c+3)}\,\,\tilde{\theta}^2(X)+\eta(X))\varphi Y.
\]
This coincides with \eqref{eq:minus}.

Thus we proved the our main theorem.

\begin{remark}{\rm The one-parameter family  
$\{A^r\}$ coincides with 
the one obtained in \cite[Theorem 1.2 (ii)]{CFG} under the choice 
$\lambda_1=\lambda_2=(c+3)/2$ and $\lambda_3=2$. Note that 
in \cite[Theorem 1.2]{CFG} and \cite[Theorem 1.3]{CFG}, the left 
invariance of $S$ is (implicitly) assumed.
}
\end{remark}

\subsection{The Heisenberg group}
In Theorem \ref{thm:MainTheorem},
by choosing $c=-3$,
we retrieve the classification of 
homogeneous Riemannian structures on 
$(\mathrm{Nil}_3,4g)$ due to Tricerri and Vanhecke \cite[Theorem 7.1]{TV}.
It should be remarked that 
the coset space representation $\mathrm{Nil}_3/\{\mathsf{e}\}$ 
corresponds to the Tanaka-Webster connection. Namely 
the Tanaka-Webster connection 
on $\mathrm{Nil}_3$ 
is characterized as the only 
canonical connection (Ambrose-Singer connection) 
corresponding to the 
coset space representation $\mathrm{Nil}_3/\{\mathsf{e}\}$. Moreover 
the Tanaka-Webster connection coincides with Cartan-Schouten's 
$(-)$-connection.

\subsection{The unit $3$-sphere}\label{sec:Abe}
Finally we discuss the homogeneous Riemannian structures 
of the unit $3$-sphere $\mathbb{S}^3$. 
In our setting and notation, Abe's classification is 
reformulated in the following manner:
\begin{theorem}[\cite{Abe}]
The homogeneous Riemannian structures on 
the unit $3$-sphere are classified as follows{\rm:}
\begin{enumerate}
\item $S(X)Y=-r\,\mathrm{d}V(X,Y)$ for some constant 
$r\geq 0$, $r\not=1$. The corresponding 
coset space representation is 
\[
\mathbb{S}^3=\mathrm{SO}(4)/\mathrm{SO}(3)=
(\mathrm{SU}(2)\times\mathrm{SU}(2))/\mathrm{SU}(2).
\]
The homogeneous Riemannian structure is of type $\mathcal{T}_3$.
\item $S(X)Y=A^{r}(X)Y$ for some $r\in\mathbb{R}$ with $r\not=1$. 
The homogeneous Riemannian structure is of type 
$\mathcal{T}_2\oplus\mathcal{T}_3$. 
It is of type $\mathcal{T}_2$ if and only if $r=-2$. 
The corresponding 
coset space representation is 
\[
\mathbb{S}^3
=(\mathrm{SU}(2)\times\mathrm{U}(1))/\mathrm{U}(1)=
\mathrm{U}(2)/\mathrm{U}(1).
\]
\item $S(X)Y=A^{1}(X)Y=-\mathrm{d}V(X,Y)$. 
The homogeneous Riemannian structure is of type 
$\mathcal{T}_3$. 
The corresponding 
coset space representation is 
\[
\mathbb{S}^3=\mathrm{SU}(2)/\{\mathsf{e}\}.
\]
\end{enumerate}
\item 
\end{theorem}
Although the homogeneous Riemannian structure $A^r$ is 
a homogeneous contact Riemannian structure for any $r\in\mathbb{R}$, 
the homogeneous Riemannian structure 
$-r\mathrm{d}V$ is a homogeneous contact Riemannian structure when and 
only when $r=1$. Note that $\dim \mathrm{SO}(4)=
\dim (\mathrm{SU}(2)\times\mathrm{SU}(2))=6$. On the other hand, 
$\dim \mathrm{Aut}(\mathbb{S}^3)=4$ \cite{Ta0}. Now we retrieve
the following classification \cite[Theorem 5.3]{GO} due to 
Gadea and Oubi{\~n}a.

\begin{corollary}[\cite{GO}]\label{cor:GO}
The set of all homogeneous contact Riemannian structures
on the unit $3$-sphere $\mathbb{S}^3$ 
is given by
$\{A^{r}\
\vert \ r\in \mathbb{R}\}$. 
\end{corollary}
Combining this classification with our main theorem, we obtain the 
following corollary.

\begin{corollary}
The set of all homogeneous contact Riemannian structures
on a Sasakian space form $\mathscr{M}^3(c)$ coincides 
with the set $\mathcal{S}$ of all homogeneous Riemannian structures on 
$\mathscr{M}^3(c)$ described in Theorem {\rm \ref{thm:MainTheorem}} and 
Corollary {\rm\ref{cor:GO}}.
\end{corollary}
Let us consider the 
Boothby-Wang fibration 
$\pi:\mathscr{M}^3(c)\to\mathscr{M}^2(c+3)$ of a 
Sasakian space form onto the space form $\mathscr{M}^2(c+3)$.
Then we can project down a homogeneous contact Riemannian structure 
$S$ to $\mathscr{M}^2(c+3)$. 
More precisely, 
we can define a tensor filed $\bar{S}$ on 
$\mathscr{M}^2(c+3)$ by 
\[
\bar{S}(\bar{X})\bar{Y}=\pi_{*}(S(\bar{X}^{\mathsf h})\bar{Y}^{\mathsf h})
\] 
for any vector fields
$\bar{X}$, $\bar{Y}$ on $\mathscr{M}^2(c+3)$.   
Here the superscript 
$\mathsf{h}$ denotes the horizontal lift. The 
tensor field $\bar{S}$ is a homogeneous Riemannian structure on 
$\mathscr{M}^2(c+3)$. 
This homogeneous Riemannian structure is 
called the \emph{reduced homogeneous 
Riemannian structure} (see \cite[Corollary 6.1.6]{CalLo}). 
If we choose $S=A^r$, then we have $\bar{S}=0$. 
Note that $\mathscr{M}^2(c+3)$ has only 
trivial homogeneous Riemannian structures if $c+3\geq 0$. 
We can interpret this fact as 
$A^r$ is derived from the trivial homogeneous structures 
of $\mathscr{M}^2(c+3)$. 

On the other hand, in case $c+3<0$, 
we have three coset space representations 
of $\mathscr{M}^3(c)$;
\[
(\widetilde{\mathrm{SL}}_2\mathbb{R}\times\mathrm{SO}(2))/\mathrm{SO}(2),
\quad 
\widetilde{\mathrm{SL}}_2\mathbb{R}/\{\mathsf{e}\},
\quad 
\tilde{G}(c)/\{\mathsf{e}\}.
\]
In the first and two representations, homogeneous 
Riemannian structures are members of Okumura's family. Thus 
these homogeneous 
Riemannian structures are derived from the trivial 
homogeneous Riemannian structures of $\mathbb{H}^2(c+3)$.
On the other hand, in the last case, 
The Lie algebra $\mathfrak{g}\cong\mathfrak{ga}(1)\oplus
\mathbb{R}$ is the $1$-dimensional extension of $\mathfrak{ga}(1)$. 
The Cartan-Schouten's $(-)$-connection of $\tilde{G}(c)$ is 
the extension of that of the solvable Lie group model 
of $\mathbb{H}^2(c+3)$. In this sense, 
the Cartan-Schouten's $(-)$-connection of $\tilde{G}(c)$ 
is derived from the non trivial homogeneous Riemannian 
structure of type $\mathcal{T}_1$ of $\mathbb{H}^2(c+3)$ 
explained in Theorem \ref{thm:1.5} (see Remark \ref{rem:H2}).

\end{document}